\documentclass[11pt,a4paper]{amsart}
\title[Non-contractible orbits on Riemann surfaces]{Non-contractible orbits  for Hamiltonian functions on Riemann surfaces}
\author{Hiroyuki Ishiguro}

\usepackage[dvipdfmx]{graphicx,color}
\usepackage[abs]{overpic} 
\usepackage{amsmath}
\usepackage{mathtools}
\usepackage{amsthm}
\usepackage{enumerate}
\usepackage[all]{xy}
\usepackage{here}

\usepackage{amsfonts}

\newcommand{\Rbb}{\mathbb{R}}

\newcommand{\Tbb}{\mathbb{T}}
\newcommand{\Hcal}{\mathcal{H}}

\renewcommand{\phi}{\varphi}
\renewcommand{\epsilon}{\varepsilon}
\renewcommand{\Psi}{\varPsi}
\renewcommand{\Phi}{\varPhi}
\newcommand{\im}{{\rm{Im}}}
\renewcommand{\ker}{{\rm{Ker}}}
\renewcommand{\mid}{\ \left|\right. }

\newcommand{\sub}{\subset}

\newcommand{\Zbb}{\mathbb{Z}}
\DeclareMathOperator{\supp}{supp}
\DeclareMathOperator{\Area}{Area}

\DeclareMathOperator{\Fix}{Fix}
\newcommand{\Cinf}{{C^\infty}}
\newcommand{\Cinfc}{{C^\infty_c}}
\newcommand{\meas}{{\mathfrak{M}}}

\DeclareMathOperator{\Spec}{Spec}

\newcommand{\HF}[1]{{HF^{[a,b)}}(#1;\alpha)}
\newcommand{\HFind}[2]{{HF^{[a,b)}_{#1}}(#2;\tilde{\alpha})}
\newcommand{\A}{\mathcal{A}}
\newcommand{\HFp}[1]{{HF^{[b,c)}}(#1;\alpha)}
\newcommand{\HFpind}[2]{{HF^{[b,c)}_{#1}}(#2;\tilde{\alpha})}
\newcommand{\halfopen}[1]{{[#1)}}
\newcommand{\eone}{{\epsilon_1}}
\newcommand{\parbar}[1]{\partial\overline{#1}}

\theoremstyle{definition}
\newtheorem{df}{Definition}[section]
\newtheorem{ex}[df]{Example}

\theoremstyle{plain}
\newtheorem{thm}[df]{Theorem}
\newtheorem{prp}[df]{Proposition}
\newtheorem{lem}[df]{Lemma}
\newtheorem{cor}[df]{Corollary}
\newtheorem{problem}[df]{Problem}

\theoremstyle{remark}
\newtheorem{rem}[df]{Remark}
\newtheorem{claim}[df]{Claim}

\makeatletter

\@addtoreset{equation}{section}
\makeatother

\begin{document}
\begin{abstract}
		We consider two disjoint and homotopic non-contractible embedded loops on a Riemann surface and prove the existence of a non-contractible orbit for a Hamiltonian function on the surface whenever it is sufficiently large on one of the loops and sufficiently small on the other one. This gives the first example of an estimate from above for a generalized form of the Biran-Polterovich-Salamon capacity for a closed symplectic manifold. 
\end{abstract}
\maketitle
\section{Introduction}
One of the most important problems in symplectic geometry is to find a 1-periodic orbit of  Hamiltonian equations. If such an orbit exists, it is natural to estimate the number of the orbits. A well-known problem in this area is the Arnold conjecture. It states in the homological version that the number of contractible 1-periodic orbits $ P(H;0) $ for non-degenerate Hamiltonian function $ H $ on a closed symplectic manifold is greater than or equal to the sum of the Betti numbers. This problem was solved by using Floer homology  (see \cite{Fl},\cite{HS},\cite{On},\cite{FO},\cite{LT},\cite{Ru},\cite{FO2}).

Let us denote by $ M $ the unit cotangent bundle of $ \Tbb^{n} $ or a Riemannian manifold with negative sectional curvature.  Biran, Polterovich and Salamon \cite{BPS} proved that there exists a non-contractible 1-periodic orbit for every Hamiltonian function on $ M $ that has large values on the zero section. In the case of a torus, their proof shows that the number of non-contractible 1-periodic orbits  is greater than or equal to the sum of the Betti numbers of $ \Tbb^n $.
This theorem was generalized to the case where the total space $ M $ is the unit cotangent bundle of a closed Riemannian manifold by Weber \cite{Weber}.   
Kawasaki and Orita \cite{KO} proved the case where $ M $ is the product of an annulus and $ \Tbb^{2n} $ and $ \alpha $ is a homotopy class of a loop along the annulus. 
Niche \cite{Ni}, Xue \cite{Xu} and Kawasaki \cite{Ka} also considered similar problems in various settings.

Note that the spaces considered in  all of the above works are open symplectic manifolds. 
In this paper, we prove a result that contains the case of a closed symplectic manifold. 

Let $ g $ and $ e $ be non-negative integers. We denote by $  (\overline{\Sigma}_{g,e},\omega)  $ a compact Riemann surface of genus $ g $ and  $ e $ boundary components with symplectic structure $ \omega $ and by $ \Sigma\ (=\Sigma_{g,e}) $ its interior. Let $ l_0 $ and $ l_1 $ be disjoint and homotopic non-contractible embedded loops on $ \Sigma $. Note that in this case $ l_0 $ and $ l_1 $ are bounded by an annulus (see Lemma \ref{lem:construct embedding}).
The area $ \mathrm{Area}(l_0,l_1) $ of the annulus enclosed by $ l_0 $ and $ l_1 $ is defined as follows. Consider the volume form on $ S^1\times[0,1] $ defined by $ dp\wedge dq $ with respect to the coordinate chart $ (q,p)\in S^1\times [0,1] $.
\begin{enumerate}[$ \bullet $]
	\item  If $ \Sigma\ne\Tbb^2 $, define
	\[ \Area(l_0,l_1):=\int_{S^1\times[0,1]}\Psi^*\omega, \]
	where $ \Psi:S^1\times [0,1]\to M $ is a smooth homotopy from $ l_0 $ to $ l_1 $.
	\item If $ \Sigma=\Tbb^2 $, define
\begin{align*}
	\Area(l_0,l_1):=\inf\int_{S^1\times[0,1]}\Psi^*\omega,
\end{align*}
where the infimum is taken over all smooth homotopies $ \Psi $ from $ l_0 $ to $ l_1 $ with $ \int_{S^1\times[0,1]}\Psi^*\omega>0 $.
\end{enumerate}
In the first case, the value of the area does not depend on the choice of $ \Psi $, since $ \Sigma $ is $ [l_0] $-atoroidal (see Example \ref{ex:reimann surface is atoroidal}). Note that in this case, $ \Area(l_0,l_1) $ is not necessarily positive.
In the second case, there are many choices of $ \Psi $. Therefore, we choose the one which attains the least area from all the homotopies with positive area.
We denote by $ P(H;\alpha) $ the set of 1-periodic orbits for a Hamiltonian function $ H $ whose homotopy type is $ \alpha $.
\begin{thm}\label{mainthm}
	Suppose that $ l_0 $ and $ l_1 $ are homotopic and disjoint non-contractible embedded loops on $ \Sigma $ and that $ \Area(l_0,l_1) $ is positive.
	Let $ r $ be a positive integer and  $ H\in \Hcal(\Sigma)=\Cinfc(S^1\times \Sigma) $ be a  (time-dependent) Hamiltonian function satisfying 
	\begin{align}\label{eq:condition on H}
	\inf_{S^1\times \im(l_0) }H- \sup_{S^1\times \im(l_1)}H >r\cdot\mathrm{Area}(l_0,l_1) .
	\end{align} Then the following holds.
\begin{enumerate}
	\item 	There exists a 1-periodic orbit for $ H $ in class $ \alpha=[l_0]^{-r}\in \pi_1(\Sigma) $. 
	\item If  $ \Sigma\ne\Tbb^2 $ (i.e., $ (g,e)\ne(1,0) $), then we have  $\sharp P(H;\alpha)\ge 2 $. 
	\item If $ H $ is non-degenerate for every $ x\in P(H;\alpha) $, then we have $ \sharp P(H;\alpha)\ge 4 $.
\end{enumerate}
\end{thm}
\begin{rem}
	In (1) and (2), the condition \eqref{eq:condition on H} can be weakened to $ \inf_{S^1\times \im(l_0) }H- \sup_{S^1\times \im(l_1)}H \ge r\cdot\mathrm{Area}(l_0,l_1) $. Indeed, we can use the technique in the proof of Proposition 3.3.4 in \cite{BPS}. On the other hand, the condition \eqref{eq:condition on H}  cannot be relaxed in (3). 
\end{rem}

\begin{rem}
	As a related problem to the above theorem, we can consider the number of periodic orbits in class $ \alpha $ whose period is not necessarily one. Suppose that $ T\ge1 $. If $ H $ satisfies \eqref{eq:condition on H}, $ TH  $ also satisfies the condition. Applying the above theorem, we have a 1-periodic orbit $ x\in P(T H;\alpha) $. Putting $ y(t)=x(t/T) $, we observe that  $ y $ is a $ T $-periodic orbit for $ H $ in class $ \alpha $. Therefore we  conclude that for every Hamiltonian function $ H\in \Hcal(\Sigma) $ with \eqref{eq:condition on H}, there are infinitely many periodic orbits for $ H $ in class $ \alpha $ (whose period is not necessarily one). 
\end{rem}
For a compact subset $ X $ of an open symplectic manifold $ (M,\omega) $ and $ \alpha \in \pi_1(M) $, Biran, Polterovich and Salamon defined a relative capacity by
\[ C_{BPS}(M,X;\alpha)=\inf\{c>0\mid \forall H \in \Hcal_c(M,X), P(H,\alpha)\neq \emptyset\},  \]
where 
\[ \Hcal_c(M,X)=\{H\in \Hcal(M)\mid \inf_{S^1\times X}H\ge c\}. \]
Then Theorem B in \cite{BPS} in the case of the unit cotangent bundle $ U^*\Tbb^{n} $ of a torus can be rewritten as
\[  C_{BPS}(U^*\Tbb^n,\Tbb^n;\alpha)=\sqrt{\alpha_1^2+\cdots+\alpha_n^2}, \] 
where $ \alpha=[t\mapsto (\alpha_1 t,\cdots,\alpha_n t,0,\cdots,0)] $.
Following their notation, we can rephrase our theorem by using the capacity:
\[ C(M,X,Y;\alpha)=\inf \{ c>0 \mid \forall H\in \Hcal_c(M,X,Y), P(H,\alpha)\neq \emptyset \}, \]
where $ X $ and $ Y $ are two disjoint and compact subsets of a symplectic manifold $ M $, $ \alpha\in \pi_1(M) $ and \[ \Hcal_c(M,X,Y):=\{H\in \Hcal(M) \mid \inf_{S^1\times X} H-\sup_{S^1\times Y} H \ge c\}.\]
\begin{thm}\label{thm:capacity estimate}
	Under the same assumption as in Theorem \ref{mainthm}, we have
	\begin{align*}
	C(\Sigma_{g,e},l_0,l_1;[l_0]^{-r})=r\cdot \mathrm{Area}(l_0,l_1).
	\end{align*}
\end{thm}
\begin{rem}
	More precisely, Theorem \ref{mainthm} is equivalent to the statement $ C(\Sigma_{g,e},X,Y;\alpha)\le r\cdot \mathrm{Area}(l_0,l_1) $. However, we can show that  the inverse inequality holds by constructing a Hamiltonian function as in the proof of  Corollary \ref{cor:capacity estimate}. 
\end{rem}
\begin{rem}\label{rem:product by aspherical mfd}
	Let $ N $ be a closed connected aspherical symplectic manifold with an aspherical first Chern class. Theorem \ref{mainthm} can be extended to the case where the total space is $ M=\Sigma_{g,e}\times N $, $ X=\im(l_0)\times N $, $ Y=\im(l_1)\times N $, and $ \alpha=[l_0]^{-r}\times\{0\} $ (in this case, $ M $ is $ \alpha $-atoroidal by Example \ref{ex:reimann surface is atoroidal}). The proof is the same as the case $ N=\{*\} $. 
	In addition, we have
	\[ C(\Sigma_{g,e}\times N,X,Y;\alpha)=r\cdot\mathrm{Area}(l_0,l_1). \]
\end{rem}
\begin{rem}
	Not every homotopy class $ \alpha\in\pi_1(\Sigma_{g,e}) $ is represented in the form $ \alpha= [l]^{r} $ by an embedded loop $ l $ and integer $ r $. Let us call the homotopy class written in the form an essential class. In contrast to Theorem \ref{thm:capacity estimate}, we can prove that
	\[ C(\Sigma_{g,e},X,Y;\alpha)=\infty \]
	for any non-essential class  $ \alpha $ and  any compact subsets $ X $ and $ Y $.
	Indeed, we can prove this in the following way. Fix $ c>0 $ and take a time-independent Hamiltonian function $ \inf_X H- \sup_Y H>c  $.
	Suppose that $ x\in P(H;\alpha) $. Put $ t_0=\inf \{ t>0 \mid x(0)=x(t) \} $. Since $ \alpha\ne 0 $, we have $ t_0>0 $. Since $ x'(0)= X_H(x(0))=X_H(x(t_0)) =x'(t_0) $, we have  $ [x]=[y]^k $ for some positive integer $ k $ and $ y(t):=x(t_0 t) \ (0\le t\le1)$. This contradicts to the fact that $ \alpha $ is not essential.    
\end{rem}

Biran, Polterovich and Salamon \cite{BPS} and other authors \cite{Weber,Ni,Xu,KO} prove the existence of non-contractible orbits using the filtered Floer homology for non-contractible orbits, though as an exceptional case, Kawasaki \cite{Ka} gives a proof using the Hamilton Floer homology for contractible orbits.
We also use the filtered Floer homology for non-contractible orbits. However, as a notable difference, we treat the case where the total space is a Riemann surface (therefore, not necessarily open) whereas all of the papers stated above treat the case of open symplectic manifolds. The most typical case is a torus. In fact, the proof is more difficult than the other surfaces. The most difficult point is that a torus is not $ \alpha $-atoroidal. This turns out to be an obstruction to use Pozniak's theorem, which enables us to calculate the filtered Floer homology explicitly. However, this theorem is only applicable for $ \alpha $-atoroidal manifolds.

The idea to overcome this difficulty is to cut the torus and concatenate the so obtained annulus sufficiently many times. We consider the projection from the long annulus to the torus and define a Hamiltonian function  $ G^k $ on the long annulus by smoothing near the boundary the pull-back of $ H $ by the projection. Here, $ k $ refers to approximately the half of the number of concatenations. By an argument similar to \cite{BPS}, we construct two functions $ H_0 $ and $ H_1 $ on the long annulus with $ H_1\le G^k\le H_0 $  and examine the commutative diagram induced by Theorem \ref{CFH} (2):
\[
\xymatrix{
	\HF{H_0} \ar[rd]^{\sigma_{G^k H_0}}\ar[rr]^{\sigma_{H_1 H_0}} &	 & \HF{H_1}\\
	&\HF{G^k}\ar[ru]^{\sigma_{H_1 G^k}}.& 
	}
\]
By using  Theorem \ref{CFH} (3) and Theorem \ref{Pozniak}, we prove that $ \sigma_{H_1 H_0}$ is an isomorphism and that $ \HF{H_0}\simeq H_*(S^1;\Zbb/2) $ holds. Since $ \sigma_{H_1 H_0} $ factors through $ \HF{G^k} $, we conclude that $ \HF{G^k}\neq 0 $. 
Then we obtain a non-contractible 1-periodic orbit $ x $ for $ G^k $. 
Moreover, we can prove that the projection of $ x $ is a non-contractible orbit for $ H $ if $ k $ is sufficiently large, which is the desired conclusion.

This paper is organized as follows. In Section 2, we set up notation and terminology and summarize without proofs useful properties of the Floer homology. 
In Section 3, we prove the existence of non-contractible orbits for a symplectic manifold that contains a symplectic submanifold symplectomorphic to the product of an annulus and an aspherical manifold. The proof contains the argument using the filtered Floer homology stated above. 
Section 4 provides a complete proof of Theorem \ref{mainthm}. If $ g\ne1  $ and $ e=0 $, the proof directly follows from Theorem \ref{general statement}. If $ e\ne 0 $, we need to attach a long annulus to each component of the boundary of the surface. Then we can apply Theorem \ref{general statement} to the obtained surface. In the case $ g\ne 1  $ and $ e=0 $ (i.e., $ \Sigma=\Tbb^2 $), the concatenation argument stated above ensures that the proof is also attributed to Theorem \ref{general statement}.
In Section 5, we consider possible generalizations of Theorem \ref{mainthm} to a higher dimensional torus. We show that the natural generalization using Lagrangian sub-tori does not hold true by giving a counterexample and then pose an alternative formulation. 
In Section 6, we summarize the basic properties of our relative capacity.
In Section 7, we compare our capacity with the one defined by Kawasaki using invariant measures.

\subsection*{Acknowledgment}
The author would like to thank his adviser Professor Toshitake Kohno for his helpful advice. He also thanks to Morimichi Kawasaki,  Ryuma Orita and Tomohiro Asano for fruitful discussions. He also thanks to Professor Kaoru Ono for some comments. Especially, Morimichi suggested the author to study the topics around  symplectic capacities and told me the problem on the BPS capacity of a torus. Also, the author is grateful to Tomohiro for comments that led to generalization to the torus case of Theorem \ref{mainthm} (3).
This work was supported by the Program for Leading Graduate Schools.

\section{Preliminaries}
This section contains a brief summary of basic concepts in symplectic geometry and filtered Floer homology for $ \alpha $-atoroidal symplectic manifolds. 

\subsection{Atoroidal symplectic manifolds}
Let $ (M,\omega) $ be a symplectic manifold. For a Hamiltonian function $ H\in  \Hcal(M):=\Cinfc(S^1\times M) $, the \textit{Hamiltonian vector field} $ X_H $ of $ H $ is defined by $ dH=-\omega(X_H,-) $. The flow of $ X_H $  is denoted by $ \phi_H^t $ and called the \textit{Hamiltonian flow} of $ H $ and we call its time-one map the \textit{Hamiltonian diffeomorphism} of $ H $.  A map $ x:S^1\to M $ is said to be a \textit{1-periodic orbit} for a Hamiltonian function $ H $, if $ x $ satisfies the following differential equation referred to the \textit{Hamiltonian equation}
\[ \dot{x} =X_H(x). \]
In addition, the set of 1-periodic orbits for $ H $ in homotopy class $ \alpha $ is denoted by $ P(H;\alpha) $.
 A Hamiltonian function $ H $ is said to be \textit{non-degenerate} for a 1-periodic orbit $ x $ if $ det(d\phi_H^1(x(0))-1)\ne0 $ (i.e., the linear map  $ d\phi_H^1(x(0)):T_{x(0)}M \to T_{x(0)}M $ has no eigenvalue equal to 1).
Note that if we take a Darboux chart on $ M $ denoted by $ (q,p) $, that is, $ \omega $ is given by $ \sum_{i=1}^{n}dp_i\wedge dq_i $, then the Hamiltonian equation $ \dot{x} =X_H(x) $ is written in the form:
\begin{align*}
\dot{q_i}=\dfrac{\partial H}{\partial p_i},\qquad
\dot{p_i}=-\dfrac{\partial H}{\partial q_i}.
\end{align*}
 
We call a  symplectic manifold $ (M,\omega) $ \textit{aspherical} if for every map $ u:S^2 \to M $ the integral of $ \omega $ over $ u $ vanishes: $ \int_{S^2} u^*\omega=0$. 
A symplectic manifold $ (M,\omega) $ is said to be \textit{$ \alpha $-atoroidal} if for every map $ u:S^1\to L_\alpha M $ considered as the map $ u:\Tbb^2\to M $ the integral of $ \omega $ over $ u $ vanishes: $ \int_{\Tbb^2} u^*\omega=0$, where $ L_\alpha M $ stands for the loop space of $ M $ whose loops are in class $ \alpha $. More generally, we call a closed 2-form $ \beta\in \Omega^2(M) $ aspherical (resp., $ \alpha $-atoroidal) if the integral of $ \beta $ over $ u $ vanishes, for every element $ u $ in $ \pi_2(M) $ (resp., $ \pi_1 (L_\alpha M) $). Note that $ \beta $ is aspherical if and only if $ \beta $ is $ 0 $-atoroidal.
Fix a reference loop $ z\in L_\alpha M $. 
For an $ \alpha $-atoroidal manifold $ (M,\omega) $, define the action function $ \A_H:L_\alpha M\to \Rbb$ by
\[ \A_H(x)=\int_0^1 H_t(x(t))dt -\int u^* \omega, \]
where $ u:S^1\times[0,1] \to M $ is a map bounding $ z=u(-,0) $ and $ x=u(-,1) $.
Note that $ \A_H $ is well defined because of the $ \alpha $-atoroidal assumption.
An easy computation shows that the critical points of $ \A_H $ are exactly 1-periodic orbits for $ H $ in class $ \alpha $.
We denote the set of the critical values of $ \A_H $ by $ \Spec(H;\alpha):=\A_H(P(H;\alpha)) $.
\begin{thm}\label{Oh}
	Let $ (M,\omega) $ be a closed or open convex symplectic manifold with $ \alpha $-atoroidal symplectic form, where $ \alpha \in \pi_1(M)$ and  $ H\in\Hcal(M) $. Then $ \Spec(H,\alpha) $ is a bounded subset in $ \Rbb $ whose measure is zero.
\end{thm}
\begin{proof}
	Use the argument in the proof of  \cite[Lemma 18.11]{Oh}.
\end{proof}
\begin{lem}\label{lem:product of atoroidal and aspherical}
	Let $ \beta $ be an $ \alpha $-atoroidal closed 2-form $ M $ and $ \beta' $ be an $ \alpha' $-atoroidal closed 2-form on $ N $. Then $ \pi_M^*\beta+\pi_N^*\beta' $ is an $ (\alpha,\alpha') $-atoroidal closed 2-form on $ M\times N $. Here, $ \pi_M $ and $ \pi_N $ are the projections from $ M\times N $ to $ M $ and $ N $. 
\end{lem}
\begin{proof}
	Take a map $ u:\Tbb^2\to M\times N $ with $ u(-,t)\in L_{(\alpha,\alpha')} (M\times N) $. Then
	\begin{align*}
	\int u^*(\pi_M^*\beta+\pi_N^*\beta')&=\int (\pi_M\circ u)^*\beta+ \int(\pi_N\circ u)^*\beta'
	\end{align*}
	Here, $ \pi_M\circ u$ is a loop in $L_\alpha M $ and  $ \pi_N \circ u$ a loop in $L_{\alpha'}  N $. Since $ M $ is $ \alpha $-atoroidal and $ N $ is $ \alpha' $-atoroidal, the right hand side of the above equality is zero. 
\end{proof}
Next, consider examples of $ \alpha $-atoroidal symplectic manifolds.
\begin{ex}\label{ex:reimann surface is atoroidal}\leavevmode
	\begin{enumerate}
		\item Consider a Riemann surface $ \Sigma=\Sigma_{g,e}$. If $ \Sigma\ne\Tbb^2, S^2  $, then every closed 2-form $ \beta $ is $ \alpha $-atoroidal for any $ \alpha\in\pi_1(\Sigma) $. In particular, $ \Sigma $ is $ \alpha $-atoroidal and aspherical.
		\item Let $ \omega $ be a symplectic form on $ \Sigma $ and $ \alpha'\in \pi_1(\Sigma) $. Suppose that $ (N,\omega') $ is an aspherical symplectic manifold with an aspherical first Chern class.  Put $ M=\Sigma_{g,e}\times N $ and $ \alpha=(\alpha',0)\in\pi_1(\Sigma)\times\{0\}\sub \pi_1(M) $. Then the symplectic form $ \omega\oplus \omega' $ and the first Chern class of $ M $ are both $ \alpha $-atoroidal and aspherical. 
	\end{enumerate}
\end{ex}
\begin{proof}\hspace{\columnwidth}
	\begin{enumerate}
		\item 
		We show that $ \int u^*\beta $  is zero for every smooth map $ u:\Tbb^2\to \Sigma $.
		First, suppose that $ e\ne 0 $. In this case, since $ \Sigma $ is homotopic to a graph, $ \beta $ is exact, and hence, by Stoke's theorem, the integral $ \int u^*\beta $ is zero.
		Next, suppose that $ e=0 $. Note that $ g\ge 2 $ by the assumption. By Proposition 1.6
		of \cite{Kn} (or by a more general Theorem 4.1 of \cite{Sk}),
		\begin{align*}
		\chi(\Tbb^2)\le \mathcal{G}(u)\chi(\Sigma),  
		\end{align*}
		where $ \chi $ stands for the Euler characteristic number and $ \mathcal{G}(u) $ the geometric degree of $ u $ defined by
		\begin{align*}
		\mathcal{G}(u)=\inf d_{w,D},
		\end{align*}
		where the infimum is taken over all $ w:\Tbb^2\to\Sigma $ homotopic to $ u $ and open 2-disks $ D $ on $ \Sigma $ such that $ w|_{w^{-1}D} $ is a $ d_{w,D} $-fold covering.
		By definition, we have $ \chi(\Tbb^2)=0 $, $ \chi(\Sigma)<0 $ and $ \mathcal{G}(u)\ge 0 $. Hence we conclude that $ \mathcal{G}(u)=0 $. This implies that $ u $ is homotopic to a non-surjective map $ w $. Let $ \overline{D}\sub \Sigma\setminus\im(w) $ be a closed 2-disk and put $ \Sigma'=\Sigma\setminus \overline{D} $. Then $ \beta'=\beta|_{\Sigma'} $ is an exact form and we calculate $ \int u^*\beta= \int w^*\beta'=0 $.
		\item This follows from (1) and Lemma \ref{lem:product of atoroidal and aspherical}.
		\qedhere  
	\end{enumerate}
\end{proof}

\subsection{Filtered Floer homology for $ \alpha $-atoroidal symplectic manifolds}
The key tool used in the proof of Theorem \ref{mainthm} is the filtered Floer homology. This is considered as the Morse homology filtered by the function $\A_H$ on the loop space $ L_{\alpha}M $. The Morse homology $HM_{*}^{[a,b)}(F)$ filtered by $ F $ is defined for a Morse function $F$ on a closed manifold $ M $ whose critical value is not equal to $a$ nor $ b$. This homology is defined by a chain complex generated by critical points with critical values in the interval $ [a,b) $. It is known to be isomorphic to the relative homology
\[HM_{*}^{[a,b)}(F)\simeq H_{*}(F^{-1}([a,b),F^{-1}(a))).\]
Similarly, we define the filtered Floer homology for non-contractible 1-periodic orbits using a filtration given by the action functional.
This version of Floer homology has been studied and used in several papers, for example, \cite{BPS}, \cite{Gurel} and \cite{KO}.

\begin{df}
	An open symplectic manifold is said to be \textit{convex}  if there exists a compact symplectic manifold $ \overline{M} $ with boundary $ \partial \overline{M} $ with $ \overline{M}\setminus \partial \overline{M} =M $ and a Liouville vector field $ X $ (i.e., a vector field with $ \mathcal{L}_{X}\omega=\omega $) defined on an neighborhood of $ \partial \overline{M} $ such that $  X $ points outward along $ \parbar{M} $. 
\end{df}
\begin{rem}\label{rem:product of convex}
	Note that the product of open convex manifolds is not necessarily convex (for example, Exercise 9.2.13 in \cite{MS}). Note also that the product of an open convex manifold and a closed symplectic manifold is not convex as can be seen in the next remark.
	However, the Floer homology of such a manifold can still be defined. Indeed, Frauenfelder and Schlenk \cite{FS} defined Floer homology for symplectic manifolds with corners.
\end{rem}
\begin{rem}
	Let us show that the product of an open convex manifold $ (M^{2m},\omega_M) $ and closed symplectic manifold $ (N^{2n},\omega_N) $ is not convex with respect to $\omega= \omega_M\oplus\omega_N $. Indeed, suppose that there exists a Liouville vector field $ X $ pointing outward along $ \parbar{M}\times N $ defined on a neighbourhood $ U \times N $ of the boundary. Then $ \omega=\mathcal{L}_X \omega=d\iota_X \omega $. Take $ x\in U $ and define $ i_x:N\to U\times N $ by $ y\mapsto (x,y) $. Then $ \omega_N=i^*_x\omega=di_x^*(\iota_X\omega) $. Although the left hand side of the first equality defines a non-trivial second cohomology class in $ H^2(N) $, the class of the right hand side of the second equality is zero. This leads to a contradiction.
\end{rem}

From now on, we assume that $ (M,\omega) $ is a closed or open convex symplectic manifold and that the symplectic form $ \omega $ and the first Chern class $ c_1(T^*M) $ are both $ \alpha $-atoroidal for a fixed non-trivial free homotopy class $ \alpha $ in $ \pi_1(M) $. 
\begin{df}
	Suppose that $ H\in \Hcal(M)$ is non-degenerate for every 1-periodic orbit in $ \alpha $ and that $ -\infty\le a<b\le \infty $ satisfy $ a,b\notin \Spec(H;\alpha) $. We define a $ \Zbb/2 $-vector space called the Floer chain complex by  
\begin{align*}
 CF^{a}(H;\alpha) &:=\Zbb/2\cdot \left\{x\in P(H,\alpha)\mid \A_H(x)<a\right\},\\
 CF^{[a,b)}(H;\alpha)&:=CF^b(H;\alpha)/CF^a(H;\alpha). 
\end{align*}
\end{df}
There is a homomorphism $ \partial:CF^{[a,b)}(H;\alpha)\to CF^{[a,b)}(H;\alpha) $ defined by counting negative gradient trajectories $ u:\Rbb\times S^1 \to M $ of $ \A_H $ modulo $ \Rbb $ action $ (\tau,u)\mapsto \left( (s,t)\mapsto u(s+\tau,t)\right) $. In addition, the formula $ \partial\circ \partial=0 $ holds (see, for example, \cite{Fl}). Then we obtain a homology: 
\[ \HF{H} = \ker(\partial)/\im(\partial).\]
If $ H $ is degenerate, by a small perturbation we obtain a non-degenerate Hamiltonian function $ H' $. 
We can prove that the isomorphism class of the Floer homology $ \HF{H'} $ do not depend on the choice of $ H' $.
Put $ P^\halfopen{a,b}(H,\alpha):=\{x\in P(H;\alpha)\mid a<\A_H(x)<b \} $. By definition, the next proposition is obvious.
\begin{prp}\label{Floer_fundamental}
	Let $ H\in \Hcal(M) $ be a Hamiltonian function with $ a,b\notin \Spec(H;\alpha) $.
	If the Floer homology is non-trivial, i.e., $ HF^{[a,b)}(H;\alpha)\neq 0 $, then there exists a periodic orbit $ x $ for  $ H $ which satisfies 
	$ [x]=\alpha $ and $ a<\A_H(x)<b $. Moreover, if $ H $ is non-degenerate, the number $\sharp P^\halfopen{a,b}(H,\alpha) $ of such 1-periodic orbits is greater than or equal to the dimension of $ \HF{H} $ over $ \Zbb/2 $.
\end{prp}
The filtered Floer homology has good properties listed in the following theorem. 
\begin{thm}[\cite{FH},\cite{CFH},\cite{BPS},\cite{Gurel}]\label{CFH}
	Suppose that $ (M,\omega) $ is an closed or open convex symplectic manifold and that the symplectic form and the first Chern class are $ \alpha $-atoroidal for a fixed non-trivial element  $ \alpha $ in $ \pi_1(M) $. Let $ -\infty\le a<b\le \infty $.
	\begin{enumerate}[(1)]
		\item (Monotone homotopy). Let $ H_0, H_1\in \Hcal(M) $ be two Hamiltonian functions satisfying $ a,b\notin \Spec(H_i;\alpha) \ (i=0,1)$, and $ H_0\ge H_1 $. Then we have a map
		\[ \sigma_{H_1 H_0}: \HF{H_0}\to \HF{H_1}. \]
		\item (Functoriality). Let $ H_0, H_1, H_2 \in \Hcal(M) $ be three Hamiltonian functions satisfying $ a,b\notin \Spec(H_i;\alpha) \ (i=0,1,2)$, and $ H_0 \ge H_1 \ge H_2 $. Then we have the following commutative diagram:
		\[
		\xymatrix{
			\HF{H_0} \ar[rd]^{\sigma_{H_1 H_0}}\ar[rr]^{\sigma_{H_2 H_0}} & & \HF{H_2}\\
			&\HF{H_1}\ar[ru]^{\sigma_{H_2 H_1}}&						
		}
		\]
		\item Let $ H_0, H_1\in \Hcal(M) $ be two Hamiltonian functions satisfying $ a,b\notin \Spec(H_i;\alpha) \ (i=0,1)$. Assume that there exists a homotopy $ H_s $ from $ H_0 $ to $ H_1 $ that satisfies the following:
		\begin{enumerate}[(3.a)]
			\item $ \dfrac{\partial H_s}{\partial s}\le 0 $,
			\item $ a,b\notin \Spec(H_s;\alpha) $ for every $ s $.
		\end{enumerate}
		Then the map $ \sigma_{H_1 H_0}: \HF{H_0}\simeq\HF{H_1} $ induced by (1) is an isomorphism.
	\end{enumerate}
\end{thm}
Although \cite{FH} and \cite{CFH} treat the contractible case, the proof goes through in this case as well.
Other than the above, we also use a theorem first proved by Pozniak in the Lagrangian intersection context and by Biran, Polterovich and  Salamon \cite[Theorem 5.2.2]{BPS} in the Hamilton Floer context for exact symplectic manifolds. This theorem is stated in \cite{Gurel} under the same condition as our case. 
A subset of the loop space $ P\subset P(H;\alpha) $ is called a Morse-Bott manifold of periodic orbits if  $ P $ satisfies the following conditions: 
\begin{enumerate}
	\item $ P(0)=\{x(0)\mid x\in P\} $ is a compact submanifold of $ M $.
	\item $ \forall x\in P $, $ T_{x(0)}P(0) \simeq T_x P:=\{\xi \in \Gamma(x^*TM) \mid \nabla_t \xi = \nabla_\xi X_H \} $, $ \nabla $ stands for the Levi-Civita connection with respect to a Riemannian metric compatible with $ \omega $. 
\end{enumerate}

\begin{thm}[\cite{Pozniak}, \cite{BPS}]\label{Pozniak}
	Let $ (M,\omega) $ be an closed or open convex symplectic manifold and suppose that the symplectic form and the first Chern class are both $ \alpha $-atoroidal and aspherical.  Assume $ -\infty\le a<b\le \infty $ and that  $ P^\halfopen{a,b}(H,\alpha)$ is a connected Morse-Bott manifold of periodic orbits. Then $ \HF{H}\simeq H_*(P^\halfopen{a,b}(H,\alpha);\Zbb/2) $.
\end{thm}
The proof is exactly in the same way as in \cite{BPS}.

\section{Main results}
 We first state a more general theorem before we prove Theorem \ref{mainthm}.
Throughout this section, we assume that  $ (M^{2n},\omega) $ is an closed or open convex symplectic manifold and that the symplectic form and the first Chern class are both $ \alpha $-atoroidal and aspherical for a fixed non-trivial element  $ \alpha $ in $ \pi_1(M) $. We fix a reference loop $ z:S^1\to M $ in $ \alpha $ as in the previous section.
Let $ N $ be a connected aspherical symplectic manifold  and put $ W:=(\Tbb^1\times[0,R] \times N,\omega_W) $, where $ \omega_W $ is a symplectic form given by $ \omega_W=dp\wedge dq\oplus\omega_N $ with respect to the natural coordinate chart $ (q,p)\in \Tbb^1\times [0,R] $. Fix $ r\in \Zbb_+ $ and put $ C=rR $.
\begin{thm}\label{general statement}
	Suppose that there exists a symplectic embedding $ \Psi:W\to M$ such that $ \alpha= \Psi_*([t\mapsto (-rt,0,*_N)]) $ and that the reference loop $ z $ is given by $z(t)= \Psi(-rt,0,*_N) $, where  $ \Psi_*:\pi_1(W)\to \pi_1(M) $ and $ *_N\in N $. Put $ X=\Psi(\Tbb^1\times\{0\}\times N) $ and $  Y=\Psi(\Tbb^1\times\{R\}\times N)  $. Take $ H\in\Hcal(M) $ satisfying $ \inf_X H -\sup_Y H>C $.
	If $ M $ is open, we also assume the following condition: 
	\begin{equation}\label{eq:open condition}
		\normalfont
		\begin{cases}
			\parbox{\columnwidth-70pt}{
			For every path  $  u:[0,1]\to L_\alpha \overline{M} $ with $ u(0)=z$ and $u(1)\sub \parbar{M}$, we have $ -\int u^*\omega < c_H $ or $ -\int u^*\omega > c_H', $ where $ c_H $ and $ c_H' $ are constants depending on $ H $ defined by \ref{def:cH}.
			}
		\end{cases}
	\end{equation}
	Then there exist $ a<b<c $ such that $ \HF{H} $ and $ \HFp{H} $ have a direct summand isomorphic to $  H_*(W;\Zbb/2) $ and we have $ \sharp P(H;\alpha)\ge 2 $.
	Moreover, if $ H $ is non-degenerate for every $ x\in P(H;\alpha) $, then we have $ \#P(H;\alpha)\ge 2b(W)=4b(N) $, where $ b $ stands for the sum of Betti numbers over $ \Zbb/2 $.
\end{thm}
Note that by extending $ \Psi $ slightly, we may assume that the symplectic embedding is defined on $ W':= \Tbb^1\times[-\tau,R+\tau] \times N $ for some $ 0<\tau<R $. In addition, if $ M $ is open, we may also assume that there exists a neighborhood $ U $ of the boundary $\parbar{M}$ satisfying the following conditions: 
\begin{enumerate}[$ \bullet $]
	\item  $ \Psi(W') $ does not intersect with  $ U $.
	\item  $ H|_U $ is zero.
	\item  The condition \eqref{eq:open condition} holds for every path $ u:[0,1]\to L_\alpha \overline{M} $ with $ u(0)=z $ and $ u(1)\sub U $.
\end{enumerate}
\begin{cor}\label{cor:capacity estimate}
If there is no path $ u:[0,1]\to L_\alpha \overline{M}$ with $ u(0)=z\text{ and }u(1)\sub \parbar{M} $, in particular, \eqref{eq:open condition} is always satisfied, then
\[ C(M,X,Y;\alpha) =C.\] 
\end{cor}
\begin{proof}
	By Theorem \ref{general statement}, $ C(M,X,Y;\alpha) \le C $. Let $ 0<\delta<C $ and choose a  function  $ f\in \Cinfc((-\tau,R)) $  with  $ -r<f'(x) $ and  $ f(0)=C-\delta $.
	 Define a Hamiltonian function $ H\in \Hcal(M) $ by 
	 \[ H(x)=\begin{cases}
	 f(p)& (x=\Psi(q,p,y)\in \Psi(\Tbb\times(-\tau,R)\times N)),\\
	 0& \mathrm{otherwise}.
	 \end{cases} \]
	 Then we have $ P(H;\alpha)=\emptyset $, which implies $ C(M,X,Y;\alpha) \ge C-\delta $. Letting $ \delta\to 0 $, we have $ C(M,X,Y;\alpha)\ge C $.
\end{proof}

\begin{proof}[Proof of Theorem \ref{general statement}]
	Our goal is to prove the following proposition:
	\begin{prp}\label{H0H1_properties}
		There exist Hamiltonian functions $ H_0 , H_1 \in \Hcal(M)$ and constants $a<b<c $ satisfying the following conditions: 
		\begin{enumerate}[(i)]
			\item $ a,b,c\notin \bigcup_{i=0,1} \Spec(H_i;\alpha)\bigcup\Spec(H;\alpha)  $.
			\item $ H_0\ge H\ge H_1 $.
			\item $ \HF{H_i}\simeq H_*(W;\Zbb/2)\simeq\HFp{H_i}\ (i=0,1)$.
			\item $ \sigma_{H_1,H_0}:\HF{H_0}\to\HF{H_1} $ and $ \sigma_{H_1,H_0}:\HFp{H_0}\to\HFp{H_1} $ are isomorphisms of $ \Zbb/2 $-vector spaces.
		\end{enumerate}
	\end{prp}
	Let us confirm that Theorem \ref{general statement} is established if we prove this proposition. By Theorem \ref{CFH} (i) and (ii), we have a commutative diagram
	\[
	\xymatrix{
		\HF{H_0} \ar[rd]^{\sigma_0}\ar[rr]^{\sigma_{H_1 H_0}} & & \HF{H_1}\\
		&\HF{H}\ar[ru]^{\sigma_1},&						
	}
	\]
	where $ \sigma_0=\sigma_{H H_0} $  and $ \sigma_1= \sigma_{H_1 H} $.
	By (iv), $ \sigma_0 $ is injective. Therefore there is a $ \Zbb/2 $-vector space $ \mathcal{V} $ and the isomorphisms
	\[ \HF{H}\simeq\HF{H_1}\oplus \mathcal{V}\simeq H_*(W;\Zbb/2)\oplus \mathcal{V} \] 
	hold. Here, we used (iii) in the second isomorphism. In the same way, $ \HFp{H} $ has a direct summand isomorphic to $ H_*(W;\Zbb/2) $.
	By Theorem \ref{Floer_fundamental}, we have $ x ,x'\in P(H;\alpha) $ with $ a<\A_H(x)<b<\A_H(x')<b $. This inequality shows $ x\ne x' $, which gives $ \sharp P(H;\alpha)\ge2 $. 
	
	Now, suppose that $ H $ is non-degenerate. Then again by Theorem \ref{Floer_fundamental}, $ \sharp P^\halfopen{a,b}(H;\alpha)\ge \dim_{\Zbb/2}H_*(W;\Zbb/2) $ and $ \sharp P^\halfopen{b,c}(H;\alpha)\ge \dim_{\Zbb/2}H_*(W;\Zbb/2) $ hold.
	Therefore 
\begin{align*}
	\sharp P(H;\alpha)&\ge \sharp P^\halfopen{a,b}(H;\alpha)+\sharp P^\halfopen{b,c}(H;\alpha)\\
	&\ge2\dim_{\Zbb/2}H_*(W;\Zbb/2)=2b(W).\qedhere
\end{align*}
\end{proof}
	
\subsection{The definitions of $ H_0$ and $H_1 $}
We first construct Hamiltonian functions $ H_0 $ and $ H_1 $ on $ M $. One can easily prove the following lemma:
\begin{lem}
	There exists $ \eone $ with $\tau> \eone>0 $ such that 	
	\begin{align}\label{eq:eone}
	m_X-S_Y>C,
	\end{align}
	where
	\begin{align}
		m_X:=\inf_{S^1 \times \Tbb^1\times (-\eone,\eone)\times N} H, \quad 
		S_Y:=\sup_{S^1 \times \Tbb^1\times (R-\eone,R+\eone)\times N} H.
	\end{align}
\end{lem}

\begin{df}\label{def:mu}
 	We choose a function $ \mu \in \Cinf(\Rbb) $ satisfying the following conditions:
 	\begin{enumerate}[$ \bullet $]
 		\item $ \mu(x)=\begin{cases}
 		1 & (x\le 0),\\
 		0 & (1\le x),
 		\end{cases} $
 		\item $ \mu'(x) \le 0 $,
 		\item $ \mu''(x) \begin{cases}
 		<0 &(0<x<1/2),\\
 		=0 & (x=1/2),\\
 		>0 & (1/2<x<1).
 		\end{cases}$
 		\item $ \mu(x)=1-\mu(1-x) $.
 	\end{enumerate}
\end{df}
\begin{rem}
 	For example, we can define $ \mu $ by
 	\[ \mu(x)=\dfrac{\phi(x)}{\phi(x)+\phi(1-x)} ,\quad \phi(x)=\begin{cases}
 	0& (x\le 0),\\
 	e^{-1/x} & (0<x).
 	\end{cases} \]
\end{rem}
\begin{figure}
	\centering
	\def \svgwidth{0.5\columnwidth}
	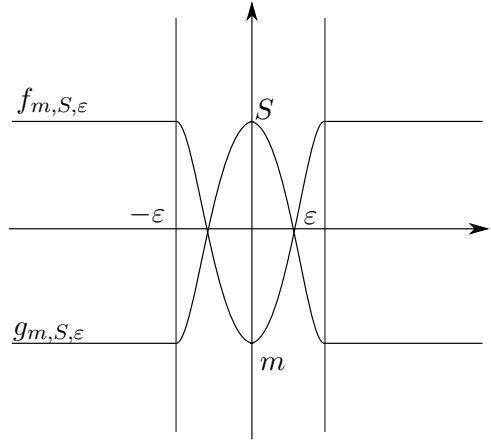
	\caption{graphs of $ f_{s,S,\epsilon} $ and $ g_{s,S,\epsilon} $ }
	\label{fg:f,g}
\end{figure}
\begin{df}
	For $ m,S,\epsilon\in\Rbb $ with $ m<S$ and $\epsilon>0 $, define functions $ g_{m,S,\epsilon},f_{m,S,\epsilon}\in\Cinf(\Rbb) $ by the following equalities (see also Figure \ref{fg:f,g})
	\begin{align}\label{def:g}
		g_{m,S,\epsilon}(x)&=(S-m)\mu({|x|}/{\epsilon})+m,\\
		f_{m,S,\epsilon}&=-g_{-S,-m,\epsilon} .\label{def:f}
	\end{align}
\end{df}
Set $ m_0,S_0,m_1$ and $S_1 $ to be 
\begin{align}\label{def:m0S0}
	m_0=S_Y ,\quad  S_0=\max\left\{\sup_M H,-\inf_M H+S_Y+m_X \right \} , \\
	\label{def:m1S1}
	  m_1=\min\left\{ \inf_M H,-\sup_M H+S_Y+m_X\right\},\quad S_1=m_X. 
\end{align}
Note that with these definitions we have
\begin{align}\label{eq:h0h1symmetry}
	S_1-m_1=S_0-m_0.
\end{align}
Note also that an easy computation shows the following estimates:
\begin{lem}\label{lem:easy estimates for m0m1S0S1}\leavevmode
	\begin{enumerate}
		\item $ S_1\le r(R-\eone)+S_0  $.
		\item $ m_1<m_0 $
		\item $ r\eone+m_1\le C+m_0 $.
		\item $ S_1\le S_0 $
		\item $  S_0\le 3S_H$.
		\item $ 3m_H \le m_1$.
	\end{enumerate}
\end{lem}
\begin{proof}\hspace{\columnwidth}
	\begin{enumerate}
		\item $ S_1\le \sup H\le S_0\le r(R-\eone)+S_0$.
		\item $ m_1\le \inf H\le S_Y <m_0$.
		\item Since $ \eone\le R $  and $ m_1<m_0 $, we have  $ r\eone+m_1\le C+m_0 $.
		\item $ S_1\le\sup H\le S_0 $.
		\item $ S_0=\max\{\sup_M H,-\inf_M H+S_Y+m_X\}\le\max\{S_H,3S_H\}=3S_H $.
		\item $  3m_H=\min\{m_H,3m_H \} \le \min\{ \inf_M H,-\sup_M H+S_Y+m_X \}=m_1$.\qedhere
	\end{enumerate}
\end{proof}\vspace{-\baselineskip}
Choose a function  $ \rho\in\Cinfc(M) $ satisfying $ \rho|_{M\setminus U}\equiv 1 $ and  $ 0\le \rho\le 1 $.
Put $ h_0=f_{m_0,S_0,\eone}$ and $ h_1=g_{m_1,S_1,\eone} $ and  define $ H_0,H_1\in\Hcal(M) $ by
\begin{align*}
	H_0(x)&=\begin{cases}
		h_0(p-R)&(x=\Psi(q,p,y)\in \Psi(W')), \\
		S_0\ \rho(x)&(x\in U),\\
		S_0 & \text{otherwise},
	\end{cases} \\
	H_1(x)&=\begin{cases}
		h_1(p) &(x=\Psi(q,p,y)\in \Psi(W')),\\
		m_1\ \rho(x)&(x\in U),\\
		m_1 & \text{otherwise}.
	\end{cases}
\end{align*} 

Fix constants $ a,b,c\in \Rbb $ which satisfies the following conditions:
\begin{align}
	m_1-1&<a<m_1\label{def:a},\\
	C+m_0&<b<S_1 \label{def:b},\\
	C+S_0&<c<C+S_0+1 \label{def:c},\\
	 a,b,c &\notin\Spec(H;\alpha)\label{def:spec_ab}.
\end{align}
One can confirm that $ a,b $ and $ c $ are well defined by Theorem \ref{Oh} and the fact $ S_1-C-m_0=m_X-C-S_Y\stackrel{\eqref{eq:eone}}{>}0 $. 

Define constants $ c_H $ and $ c_H' $ by
\begin{align}\label{def:cH}
c_H=-3S_H+3m_H-1,&& c_H'=-3m_H+C+3S_H+1.
\end{align}

\subsection{The 1-periodic orbits for $ H_0 $ and $ H_1 $ in class $ \alpha $} 
First, let us examine the 1-periodic orbits for $ H_0 $ contained in $ \Psi(W') $, or the ones for $ \Psi^*H_0 $. If we denote a point in $ W' $ by $ (q,p,y) $, the Hamiltonian equation  is
\[ \begin{cases}
\dot{q}=h_0'(p-R),\\
\dot{p}=0,\\
\dot{y}=0.
\end{cases} \]  
By the second and third equation, $ p=\mathrm{const}=p^0 \in [-\tau,R+\tau] $ and $ y=\mathrm{const}=y^0\in N $. Then 
\[ q(t)=q^0+h_0'(p^0-R)t. \]
Since $ q $ is a loop homotopic to the map $ S^1\ni t\mapsto -rt\in\Tbb^1  $, then we have $ h_0'(p^0-R)=-r $. This implies $ p^0=R-s^i\  (i=0,1)$, where $ 0<s^0<s^1<\eone $ are two solutions of $ h_0'(x)=r$  $( \Leftrightarrow h'_0(-x)=-r ) $.  With the notation $ x_0^i=(q^0-r t,R-s^i,y^0) $, we have 
\[  \A_{H_0}(x_0^i)=h_0(-s^i)-r(R-s^i). \]
Note that the value of the action functional is equal to the $ y $-intercept of the tangential line of the graph $ y=h_0(x-R) $ at $ x=R-s^i $. By comparing the graph of $ y=h_0(x-R) $ with $l_1: y=-r(x-R)+S_0 $, we have $ \A_{H_0}(x_0^1)<c $ (see Figure \ref{fg:h_0_estimate}). Similarly, by comparing the graph of $ y=h_0(x-R) $  with $l_2:y=-r(x-(R-\eone))+S_0 $, $l_3: y=-r(x-R)+m_0$ and $ l_4: y=-rx+m_0 $ and using the convexity of $ h_0 $, we have $ a<\A_{H_0}(x_0^0)<b $ and $ b<\A_{H_0}(x_0^1)<c  $. We used Lemma \ref{lem:easy estimates for m0m1S0S1} (1) and (2) in the estimates.
\begin{figure}
	\centering
	\def \svgwidth{0.7\columnwidth}
	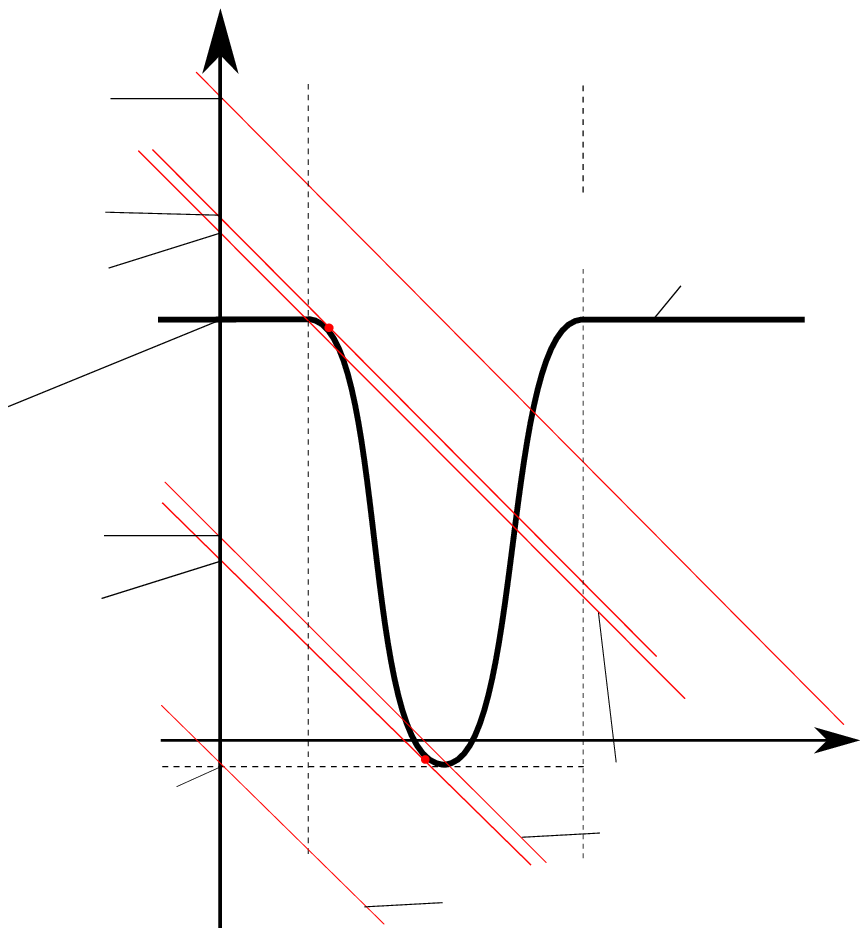
	\caption{The graphs of $h_0,l_1,l_2,l_3$ and $l_4$}
	\label{fg:h_0_estimate}
	\vspace{1em}
	\def \svgwidth{0.7\columnwidth}
	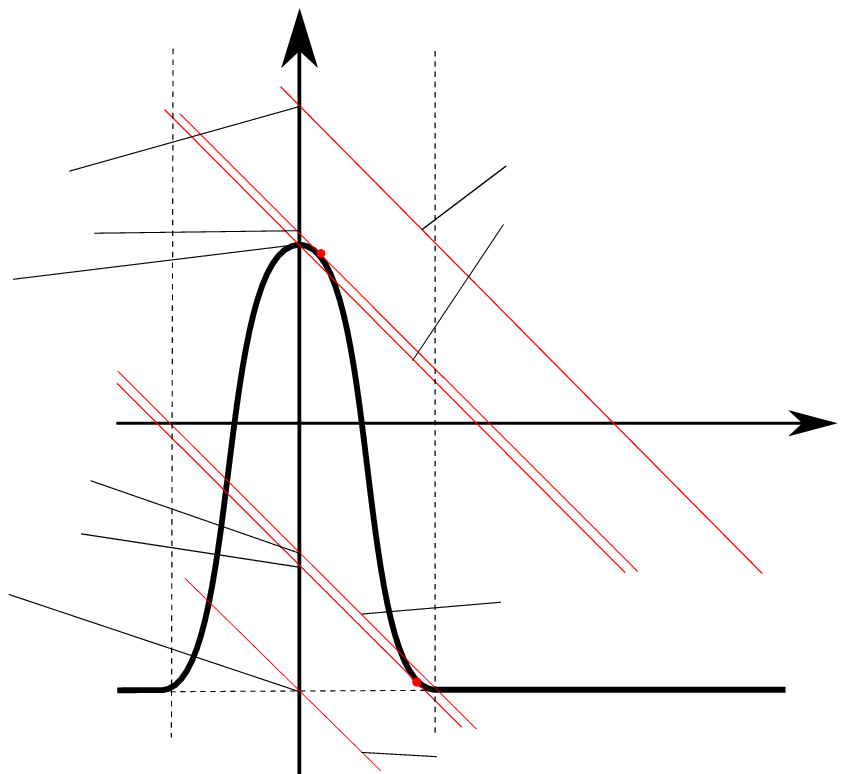
	\caption{The graphs of $h_1,l_1',l_2',l_3'$ and $l_4'$}
	\label{fg:h_1_estimate}
\end{figure}

Next, we consider the 1-periodic orbits for $ H_1 $ contained in $ \Psi(W') $. The Hamiltonian equation  is 
\[ \begin{cases}
\dot{q}=h_1'(p)  ,\\
\dot{p}=0,\\
\dot{y}=0.
\end{cases} \]  
By the second and third equation, $ p=\mathrm{const}=p^0 \in [-\tau,R+\tau] $ and $ y=\mathrm{const}=y^0\in N $. Then 
\[ q(t)=q^0+h_1'(p^0)t \]
Since $ q $ is homotopic to $ t\mapsto -rt $, we have $ h_1'(p^0)=-r $. Then $ p^0=s^i\  (i=0,1)$. Note that $ s^0$ and $s^1 $ are two solutions of $ h_1'(x)=-r$ by \eqref{eq:h0h1symmetry}.  With the notation $ x_1^i=(q^0-r t,s^i,y^0) $, we have 
\[  \A_{H_1}(x_1^i)=h_1(s^i)-rs^i \]
Note that the value of the action functional is equal to the $ y $-intercept of the tangential line of the graph $ y=h_1(x) $ at $ x=s^i $. Then by comparing the graph of $ y=h_1(x) $ with $l_1': y=-r(x-\eone)+S_1 $, $ l_2': y=-rx+S_1 $, $ l_3':y=-r(x-\eone)+m_1 $ and $ l_4':y=m_1 $, we have $ a<\A_{H_1}(x_1^1)<b $ and $ b<\A_{H_1}(x_1^0)<c $ (see Figure \ref{fg:h_1_estimate}). We used Lemma \ref{lem:easy estimates for m0m1S0S1} (3) and (4) in the estimates.

Finally, we examine the 1-periodic orbits for $ H_i $ contained in $ M\setminus \Psi(W') $. Since $ H_i $ is constant on $ M\setminus(\Psi(W')\cup U) $, every orbit contained in this subset is a constant loop, hence we only need to consider the orbits in $ U $. 
\begin{lem}\label{lem:boundary estimate}
	Suppose that $ G\in \Hcal(M) $ satisfies $ m_1\le G|_{S^1\times U}\le S_0 $. Then for every $ x\in L_\alpha M $ contained in $ U $, we have $ \A_G(x)\notin [a,c] $.
\end{lem}
\begin{proof}
If $ -\int u^*\omega<c_H $,
\begin{align*}
\A_{G}(x)&= \int_0^1 G(x(t))dt- \int u^*\omega, \\
&< S_0 + c_H\\
&= S_0 -3S_H+ 3m_H-1,\\
&\le m_1-1<a\qquad\text{by Lemma \ref{lem:easy estimates for m0m1S0S1} (5) and (6).}
\end{align*}
If $ -\int u^*\omega>c'_H $,
\begin{align*}
\A_{G}(x)
&> m_1+c_H' \\
&= m_1-3m_H+C+3S_H+1\\
&\ge C+S_0+1>c\qquad \text{by Lemma \ref{lem:easy estimates for m0m1S0S1} (5) and (6).} \qedhere
\end{align*}
\end{proof}
Note that $ H_i $ satisfies the condition of the above lemma and hence we have $ \A_{H_i}(x)\notin [a,c] $ for every loop $ x$ contained in $ U $.

Summarizing the above argument, we have the following proposition: 
\begin{prp}\label{estimate_actionfunctional}
	Every 1-periodic orbit $ x\in P(H_i;\alpha) $ with $ a<\A_{H_i}(x)<c $ is contained in $ \Psi(W') $ and written as
	\begin{align*}
		x_0^j&=(q^0-r t, R-s^j,y^0) &(i=0),\\
		x_1^j&=(q^0-r t, s^j,y^0)& (i=1),
	\end{align*} 
	where $ q^0\in \Tbb^1,y^0\in N$ and  $ s^j\ (j=0,1) $ are the solutions of the equation $ h_1'(x)=-r $.
	In addition, the values of the action functional can be estimated as
	\begin{align*}
		a<\A_{H_0}(x_0^0)<b,\qquad b<\A_{H_0}(x_0^1)&<c,\\
		a<\A_{H_1}(x_1^1)<b,\qquad b<\A_{H_1}(x_1^0)&<c.
	\end{align*} 
\end{prp}

\subsection{Proof of Proposition \ref{H0H1_properties}}\label{sec:H0H1_properties}
Now we prove Proposition \ref{H0H1_properties}. First, Proposition \ref{estimate_actionfunctional} implies (i). 

Second, we confirm (ii). By the assumption $ H|_U=0 $ and the definition of $ S_0 $ and $ m_1 $, we have $ H_1\le H\le H_0 $ in $ M\setminus\Psi(W') $. Then we only need to confirm the inequality in $ \Psi(W') $. If $ |p|\le \eone  $, we have $ H_1(q,p,y)\le S_1= m_X\le H(q,p)  $; otherwise $ H_1(q,p,y)=m_1\le H(q,p,y) $. Similarly, if $ |p-R|\le \eone  $, we have $ H_0(q,p,y)\ge m_0= S_Y\ge H(q,p,y)  $; otherwise $ H_0(q,p,y)=S_0 \ge H(q,p,y) $.
Thus, (ii) is proved.

Third, we examine (iii). These isomorphisms are a consequence of the Theorem \ref{Pozniak}. We only prove that $ \HF{H_0}\to\HF{H_1} $ is an isomorphism because the proofs are the same. It suffices to confirm that $ H_i $ satisfies the assumption of the theorem. Observe that $ P^\halfopen{a,b}(H_i;\alpha) $ is diffeomorphic to a closed manifold $ \Tbb^1\times N $ by Proposition \ref{estimate_actionfunctional}. Since this manifold is homotopy equivalent to $ W $, what is left is to show the following proposition:
\begin{prp}
	$ P_i=P^\halfopen{a,b}(H_i;\alpha) $ is a Morse-Bott manifold of periodic orbits.
\end{prp}
\begin{proof}
	Fix a Riemann metric $ g $ on $ M $ compatible with $ \omega $ such that $ \Psi^*g=dq\otimes dq +dp\otimes dp+ g_N $ holds for some Riemannian metric $ g_N $ on $ N $ compatible with $ \omega_N $.
	Consider the linearized Hamiltonian differential equation for vector fields $ \xi(t)=(\hat{q}(t),\hat{p}(t),\hat{y}(t) )\in T_{x(t)}M $
	along $ x $:
	\begin{equation}\label{linearized Hamiltonian differential equation}
	\nabla_{\dot{x}}\xi=\nabla_{\xi}X_{H_i}(x),\qquad \xi(0)=\xi(1),
	\end{equation}
	where $ \nabla $ stands for the Levi-Civita connection of $ M $.   We only need to show that the space of 1-periodic solutions of the above equation  has dimension $ \dim P_i= 1+ \dim N $ (cf.  \cite[Remark 5.2.1]{BPS}).
	If $ x\in P_i $, by Proposition \ref{estimate_actionfunctional}, $ \im(x)\sub\Psi(W') $ holds. Hence \eqref{linearized Hamiltonian differential equation} is written as 
	\begin{align*}
	\dot{\hat{q}}&=\begin{cases}
	h_0''(s^0)\hat{p}\quad  (i=0),\\
	h_1''(s^1)\hat{p}\quad  (i=1),
	\end{cases}&
	\dot{\hat{p}}&=0,&
	\dot{\hat{y}}&=0.
	\end{align*}
	Therefore
	\begin{align*}\label{key}
	\hat{q}(t)&=\begin{cases}
	\hat{q}^0+h_0''(s^0)\hat{p}^0 t\quad  (i=0),\\
	\hat{q}^0+h_1''(s^1)\hat{p}^0 t\quad  (i=1),
	\end{cases}&
	\hat{p}(t)&=\hat{p}^0,&
	\hat{y}(t)&=\hat{y}^0,
	\end{align*}
	where $\hat{q}^0\in T_{q^0}\Tbb^1, \hat{p}^0\in T_{p^0}(-\tau,R+\tau) $ and $ \hat{y}^0\in T_{y^0}N $.
	By the periodicity condition on $ \hat{q} $, we have $ h_0''(s^0)\hat{p}^0= h_1''(s^1)\hat{p}^0=0$.
	Since $ h_i''(s^j)\ne 0$, we have $ \hat{p}^0=0 $. This shows that the dimension of the solutions of \eqref{linearized Hamiltonian differential equation} equals to $ \dim P_i $.
\end{proof}

\begin{figure}
	\centering
	\def \svgwidth{0.7\columnwidth}
	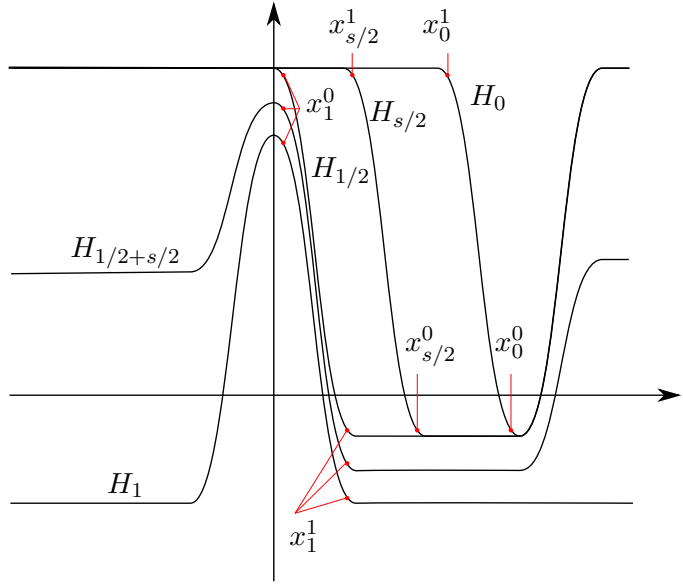
	\caption{The graphs of homotopies and its 1-periodic orbits}
	\label{fg:action_functional_of_homotopy}
\end{figure}
Finally, we prove (iv).
\begin{df}
We define $ H_{s/2} $ and $ H_{1/2+s/2}\in\Hcal(M)$ for $ 0\le s\le1 $ by
\begin{align*}
{H}_{s/2}(x)&=\begin{cases}
h_0((|p-R+sl|-sl)^+)& (x=\Psi(q,p,y)\in \Psi(W')),\\
H_0(x) & otherwise,
\end{cases}\\
H_{1/2+s/2}(x)&=(1-s)H_{1/2}(x)+sH_1(x),
\end{align*} 
where $ l:=(R-\eone)/2 $ and $ x^+:=\max\{x,0 \} $ (see Figure \ref{fg:action_functional_of_homotopy}).
\end{df}
Note that $ (|p-p^0|-l)^+=d(p,B_l(p^0)) $ holds for $ p,p^0\in\Rbb^n $, where $ d $ stands for the distance function and $ B_l(p^0) $ the closed ball with radius $ l $ centered at $ p^0 $.
Our goal is to confirm that  for $ H_{s/2} $ and $ H_{1/2+s/2} $ satisfy the two conditions (3.a) and (3.b) in Theorem \ref{CFH} (3).
The condition (3.a) is obvious.

By Lemma \ref{lem:boundary estimate}, we have $ \A_{H_s}(x)\notin[a,c] $ for every $ x\in P(H_s;\alpha) $ contained in $ U $ and $ 0\le s \le 1 $. In addition, $ H_s $ is constant on $ M\setminus(\Psi(W')\cup U) $. We only need to examine the 1-periodic orbits contained in $ \Psi(W') $ and denote by $ P'(H_s;\alpha) $ the set of such 1-periodic orbits. An easy calculation shows that $ P'(H_s;\alpha) $ is written as
\begin{align}
	P'(H_{s/2};\alpha)=\{x^j_{s/2}:=(q^0-rt,R-2sl-s^j,y^0)\mid q^0\in \Tbb, y^0\in N ,j=1,2\},\\
	P'(H_{1/2+s/2};\alpha)=\{x^j_1=(q^0-rt,s^j,y^0)\mid q^0\in \Tbb^1,y^0\in N ,j=1,2\}.
\end{align}
Note that the fourth condition in Definition \ref{def:mu} implies $ \eone=s^0+s^1 $ and hence $ x^1_{1/2}=x^0_1 $ and $ x^0_{1/2}=x^1_1 $ holds.
The value of the action functional of these orbits are estimated as (see also Figure \ref{fg:action_functional_of_homotopy})
\begin{align}
	a<\A_{H_1}(x^1_1)\le\A_{H_{1/2+s/2}}(x^1_1)\le  \A_{H_{s/2}}(x^0_{s/2})\le \A_{H_0}(x^0_0)<b,\\
	b<\A_{H_1}(x^0_1)\le \A_{H_{1/2+s/2}}(x^0_1)\le  \A_{H_{s/2}}(x^1_{s/2})\le \A_{H_0}(x^1_0)<c.
\end{align}
The first and fifth inequality is by Proposition \ref{estimate_actionfunctional}.
This proves (3.b). Then Theorem \ref{CFH} implies that $ \sigma_{H_1 H_{1/2}}\circ\sigma_{H_{1/2} H_0}=\sigma_{H_1 H_0} $ is an isomorphism,
which completes the proof of (iv).

\section{Proof of Theorem \ref{mainthm}}
In this section, we deduce Theorem \ref{mainthm} from Theorem \ref{general statement}. 
The organization of the proof can be seen in the following flow chart.
\[ \xymatrix{
	&\fbox{Lemma \ref{lem:construct embedding}}\ar[ld]_-{ \text{(I) }g\ne1\text{ and } e=0\ \ \ }\ar[d]^-{ \text{(II) }e\ne 0 }\ar[rd]^-{ \text{ (III) }\Sigma=\Tbb^2 }& \\
	\fbox{Theorem \ref{general statement}}\ar[d]& \fbox{Constructing $ \hat{\Sigma}_\rho $}\ar[d] & \fbox{Taking a covering}\ar[d]	\\
	Q.E.D.& \fbox{Theorem \ref{general statement}}\ar[d] & \fbox{Lemma \ref{general statement}}\ar[d]\\
	& Q.E.D.& \fbox{Theorem \ref{lem:calc_AH2}}\ar[d] \\
	&&Q.E.D.
}
\]
The proof of the case $ \Sigma=\Tbb^2 $ is more difficult than the others, because $ \Tbb^2 $ is not $\alpha  $-atoroidal.
The proof is divided into three parts: (I) the case $ g\ne1 $ and $ e=0 $; (II) the case $ e\ne0 $; and (III) the case $ \Sigma=\Tbb^2 $.

First, we prove a differential topological lemma.
\begin{lem}\label{lem:smooth embedding}
	Let $ l_0 $ and $ l_1 $ be disjoint and homotopic non-contractible embedded loops on $ \Sigma $. Then there exists a smooth embedding $ \Phi:S^1\times[0,1]\to \Sigma $ with $ \Phi(t,i)=l_i(t) \ (i=0,1) $.
\end{lem}
\begin{proof}
	By the transversality theorem, we have a smooth homotopy $ u:S^1\times [0,1]\to \Sigma  $ from $ l_0 $ to $ l_1 $ that is transverse to $ l_0 $ and $ l_1 $. Then the subset $ u^{-1}(l_0\cup l_1) $ is a 1-dimensional submanifold without boundary on $ \Sigma $. Denote each component by $ \gamma_i \ (i=1,\cdots,N) $. Set $ I=I(u)=\{ i\mid \gamma_i \text{ is not contractible}  \} $ and $ J=J(u)=\{ i\mid \gamma_i \text{ is contractible}  \} $. If we cut $ \Sigma $ along $ l_0 $ and $ l_1 $,  then $ \Sigma $ is divided into two or three components. Write $ \Sigma\setminus (l_0\cup l_1)=\Sigma_1\cup \Sigma_2\, (\cup \Sigma_3) $ and  $ \Sigma_k=\Sigma_{g_k,e_k} $. 
	
	We first prove that there is a smooth map $ v:S^1\times[0,1]\to \Sigma $ homotopic to $ u $ relative to the boundary and satisfying $ J(v)=\emptyset $.
	Let $ j\in J $ be a minimal element in the sense that the closed disk $ D_j  $ enclosed by $ \gamma_j $ does not contain any loops $ \gamma_{j'} \ (j'\in J) $ other than $ \gamma_j $. Let $ U_j $ be a contractible neighborhood of $ D_j $ that does not intersect with any loops other than $ \gamma_j $. For simplicity, we assume that $ u(\gamma_j)\sub l_0 $, $ u(D_j)\sub\overline{\Sigma}_1 $ and $ u(U_j\setminus D_j)\sub \Sigma_2 $. By the relative homotopy exact sequence of the pair $ (\overline{\Sigma}_1,l_0) $, we have an exact sequence $ \pi_2(\overline{\Sigma}_1) (=0)\to \pi_2(\overline{\Sigma}_1,l_0)\stackrel{f}{\to} \pi_1(l_0)(=\Zbb) $. Consider the homotopy class $ [u|_{D_j}]\in\pi_2(\overline{\Sigma}_1,l_0) $. Since $ \gamma_j $ is contractible, $ f([u|_{D_j}]) =[u|_{\gamma_j}]=0 $, and hence we have $ [u|_{D_j}]=0 $. This implies that there is a map $ u':S^1\times [0,1]\to \Sigma $ that coincides with $ u $ on $ D_j^c=S^1\times[0,1]\setminus D_j $ and homotopic to $ u $ relative to $ D_j^c $. Considering a tubular neighborhood of $ l_0 $ and using the flow of a vector field near $ l_0 $ oriented toward $ \Sigma_2 $, we can construct a map $ u'' $ equal to $ u' $ on $ U_j^c $; homotopic to $ u' $ relative to $ U_j^c $; and satisfying $ u''(U_j)\cap (l_0\cup l_1) =\emptyset $.  Then we have $ \sharp J(u'')=\sharp J(u) -1$. By an induction on $ \sharp J(u) $, we can construct a map $ v $ with $ \sharp J(v)=0 $.
	
	Next we prove the following.
	\begin{claim}
		There exists  a map $ w:S^1\times[0,1]\to \Sigma $ such that $ w(t,0)=l_0,\ w(t,1)=l_1 $ and  $ \im(w)\sub \overline{\Sigma}_k  $ for some $ k\in\{1,2,3\} $.
	\end{claim}  
	\begin{proof}
		Observe that for every $ i\in I(v) $,  $ \gamma_i $ is homotopic to $ S^1\times \{0 \} $. Additionally,  the degree of the map $ u\circ \gamma_i:S^1\to l_j\ (j=0,1) $ is $ \pm 1 $. Indeed, if  $ |\deg(u \circ \gamma_i)| \ge2 $, we can prove that $ \gamma_i $ intersects itself. In addition,  if the degree is zero,  $ u\circ \gamma_i $ is contractible. This contradicts to the assumption that $ l_0 $ is non-contractible, since $ l_0 $ is homotopic to $ u\circ \gamma_i $.  
		Now, we may assume that $ \{\gamma_i\} $ is ordered in such a way that $ D_i\sub D_j  $ holds wherever $ i<j $, where $ D_i $ stands for the annulus enclosed by $ \gamma_i  $ and $ S^1\times \{0\} $. Since $ u(\gamma_0)= l_0 $ and $ u(\gamma_N)= l_1 $, there is an $ i\in I(v) $ such that $ u(\gamma_i)=l_0 $ and $ u(\gamma_{i+1})=l_1 $. The restriction of $ u $ to the annulus enclosed by $ \gamma_{i} $ and $ \gamma_{i+1} $ gives the map $ w $ in the claim. 
	\end{proof}
	
	The lemma easily follows from the above claim.  Recall that
	\begin{align*}
	\pi_1(\overline{\Sigma}_k) = \langle a_1,\cdots , a_{g_k},b_1,\cdots,b_{g_k},z_1,\cdots,z_{e_k}\mid [a_1,b_1]\cdots[a_{g_k},b_{g_k}]=z_1\cdots z_{e_k} \rangle
	\end{align*}
	holds and consequently the inclusion map
	\begin{align*}
	\langle a_1,\cdots , a_{g_k},b_1,\cdots,b_{k},z_1,\cdots,z_{e_{k}-1}\rangle\to \pi_1(\overline{\Sigma}_k) 
	\end{align*}
	is an isomorphism. 
	Here, $ z_i $ corresponds to the homotopy class of each component of the boundary. We may assume that $ z_1=[l_0] $ and $ z_2=[l_1] $. By the above claim, we have  $ z_1=z_2^{\pm1} $. However, this happens only in the case $ (g_k,e_k)=(0,2) $. Therefore we conclude that $ \overline{\Sigma}_k $ is an annulus and the inclusion map gives the one desired.
\end{proof}
\begin{lem}\label{lem:construct embedding}
	Let $ l_0 $ and $ l_1 $ be disjoint and homotopic non-contractible embedded loops on $ \Sigma $ and suppose that $ R:=\mathrm{Area}(l_0,l_1)>0 $. Then there exists a symplectic embedding $ \Psi:S^1\times[0,R] \to \Sigma $ with
	\begin{align}\label{eq:boudary}
	\Psi(t,iR)&=l_i(t) \quad (i=0,1).
	\end{align}
\end{lem}
\begin{proof}
	By Lemma \ref{lem:smooth embedding}, there exists a smooth embedding $ \Phi: S^1\times [0,R] \to \Sigma  $ with $ \Phi(t,iR)=l_i(t) \ (i=0,1) $ and $ R=\int \Phi^*\omega $.	
	We denote by $ D $ the image of $ \Phi $. The two-dimensional compact manifold $ D $ has two symplectic forms: $ \omega_0=\omega|_D $ and $ \omega_1=\Phi_*(dp\wedge dq) $, where $ dp\wedge dq  $ is a symplectic form on $ S^1\times[0,R] $ with respect to the coordinate chart $ (q,p) $.
	Since $ \Sigma $ is a two-dimensional manifold, $ \omega_0 $ and $ \omega_1 $ are considered as volume forms. By Banyaga's theorem \cite{Ba}, for any  two volume forms $ \mu_0 $ and $ \mu_1 $ of a compact manifold $ M $ with smooth boundary that satisfies $ \int_M \mu_0=\int_M \mu_1 $, there exists a diffeomorphism $ \phi $ on $ M $ with $ \phi^*\mu_0=\mu_1 $. Appliyng this to $ M=D $, we obtain a diffeomorphism $ \phi:D\to D $ with $ \phi^*\omega_0=\omega_1 $. Then $ (\phi\circ\Phi)^*(\omega_0)=dp\wedge dq $, which means $\Psi:= \phi\circ\Phi: S^1\times [0,R] \to \Sigma $ is a symplectic embedding.
\end{proof}

\begin{proof}[Proof of Theorem \ref{mainthm}.]
\textbf{(I) Case $ g\ne1 $ and $ e=0 $ :}	
	In this case,  $ \Sigma $ is $ \alpha $-atoroidal (see Example \ref{ex:reimann surface is atoroidal}). Therefore Lemma \ref{lem:construct embedding} and Theorem \ref{general statement}
	give the desired conclusion. 

\textbf{(II) Case $ e\ne0 $ :}
	In this case, the symplectic form $ \omega $ is exact and  $ \alpha $-atoroidal. However, $ \Sigma $ may violate the condition \eqref{eq:open condition}. Therefore we take a larger symplectic manifold $ \hat{\Sigma}_\rho $ containing $ \Sigma $ defined as follows (when $ \rho=\infty $, $ \hat{\Sigma}_\rho $ is usually called the completion of $ \Sigma $).
	Let $ X $ be a Liouville vector field pointing outward along $ \parbar{\Sigma} $ defined near $ \parbar{\Sigma} $. Define a 1-form $ \lambda $ on $ \parbar{\Sigma} $ by $ \lambda:=\iota_X \omega|_{\parbar{\Sigma}} $. Then $ d\lambda= \omega|_{\parbar{\Sigma}} $ and, in particular, $ \lambda $ is non-degenerate.
	Considering the flow of $ X $, we can construct a symplectomorphism between a neighborhood of $ \parbar{\Sigma} $ and $ (\parbar{\Sigma}\times (-\epsilon,0],d(e^r \pi^*\lambda)) $ such that $ X(x,r)=\frac{\partial}{\partial r} $, where we used the coordinate chart $ (x,r)\in \parbar{\Sigma}\times (-\epsilon,0] $ and $ \pi $ is the projection $ \parbar{\Sigma}\times(-\epsilon,0]\to\parbar{\Sigma}  $.  Set $ (\hat{\Sigma}_\rho,\hat{\omega}) $ as follows
	\begin{align*}
	\hat{\Sigma}_\rho&:= \overline{\Sigma}\cup_{\parbar{\Sigma}\times \{0\}} \parbar{\Sigma}\times [0,\rho),\\
	\hat{\omega}&:=\begin{cases}
	\omega&\textrm{on } {\Sigma},\\
	d(e^r\pi^*\lambda) &\textrm{on } \parbar{\Sigma}\times (-\epsilon,\rho).
	\end{cases}
	\end{align*}
	An easy calculation shows that $ \omega $ and $ d(e^r\pi^*\lambda) $ coincide on $ \parbar{\Sigma}\times (-\epsilon,0) $.
	
	Let us prove that $ \hat{\Sigma}_\rho $ satisfies the condition \eqref{eq:open condition} if $ \rho $ is sufficiently large. Suppose that $ u(-,0)=z $ and $ u(-,1)=x_1 $, where $ x_1\sub \parbar{\hat{\Sigma}}_\rho $.
	Since $ u $ goes through some component $ x_0 $ of $ \parbar{\Sigma} $,  $ u$ is homotopic to $u_0\sharp u_1 $ relative to the boundary, where $ u_0(-,0)=z,\ u_0(-,1)=x_0=u_1(-,0) ,\  u_1(-,1)=x_1 $ and $ \sharp$ stands for the concatenation of two maps. Then we have
	\begin{align*}
	\int u^*\hat{\omega}&=\int u_0^*\omega +\int  u_1^*d(e^r\pi^*\lambda),\\
	&=\int u_0^*\omega+e^\rho \int_{x_1} \pi^*\lambda  -\int_{x_0} \lambda ,\\
	&=\int u_0^*\omega+e^\rho \int_{x_0} \lambda  -\int_{x_0} \lambda, \\
	&=\int u_0^*\omega+(e^\rho-1) \langle [\lambda], \alpha\rangle.
	\end{align*}
	We used Stokes's theorem in the second equality. Since $ |\int u_0^*\omega|\le r\cdot \mathrm{vol}(\Sigma) $ and $ \langle [\lambda],\alpha\rangle $ is not zero, we have $ \lim_{\rho\to\infty}\int u^*\hat{\omega}= \pm \infty $, which proves \eqref{eq:open condition}.
	
	By the assertion above and Lemma \ref{lem:construct embedding}, we can apply Theorem \ref{general statement} to $ \hat{\Sigma}_\rho $ and $ H $ extended by zero. Then we obtain at least two 1-periodic orbits for $ H $ on $ \hat{\Sigma}_\rho $ and the number of orbits is greater than or equal to four if $ H $ is non-degenerate. Since  $ H $ is supported in $ \Sigma $, all of the orbits are contained in $ \Sigma $. 

\textbf{(III) Case $\Sigma=\Tbb^2 $ :}
	Taking into account of Moser's trick, we may assume that $ \omega $ is a standard volume form $ \omega=\mathrm{vol}(\Sigma) dp\wedge dq $, where $ (q,p) $ is a standard coordinate chart. For simplicity, we assume that $ \mathrm{vol}(\Sigma)=1 $. There seems to be two approaches: using the Floer homology on $ \Sigma $; or taking a covering of $ \Sigma $ and using the Floer homology on the covering. In the first approach, we need to consider the Floer homology for non-atoroidal symplectic manifolds, since $ \Tbb^2 $ is not $ \alpha $-atoroidal. However this Floer homology is difficult to treat due to the non-vanishing of the Novikov ring. Therefore, we adopt the second approach. The proof is  based on the following observation. For simplicity, we temporarily assume that $ l_0(t)=(rt,0) $. To use Theorem \ref{general statement}, we cut this manifold along $ l_0 $. However, this operation may make us forget such an orbit shown in the following figure.
	\begin{figure}[H]
		\centering
		\def \svgwidth{0.4\columnwidth}
		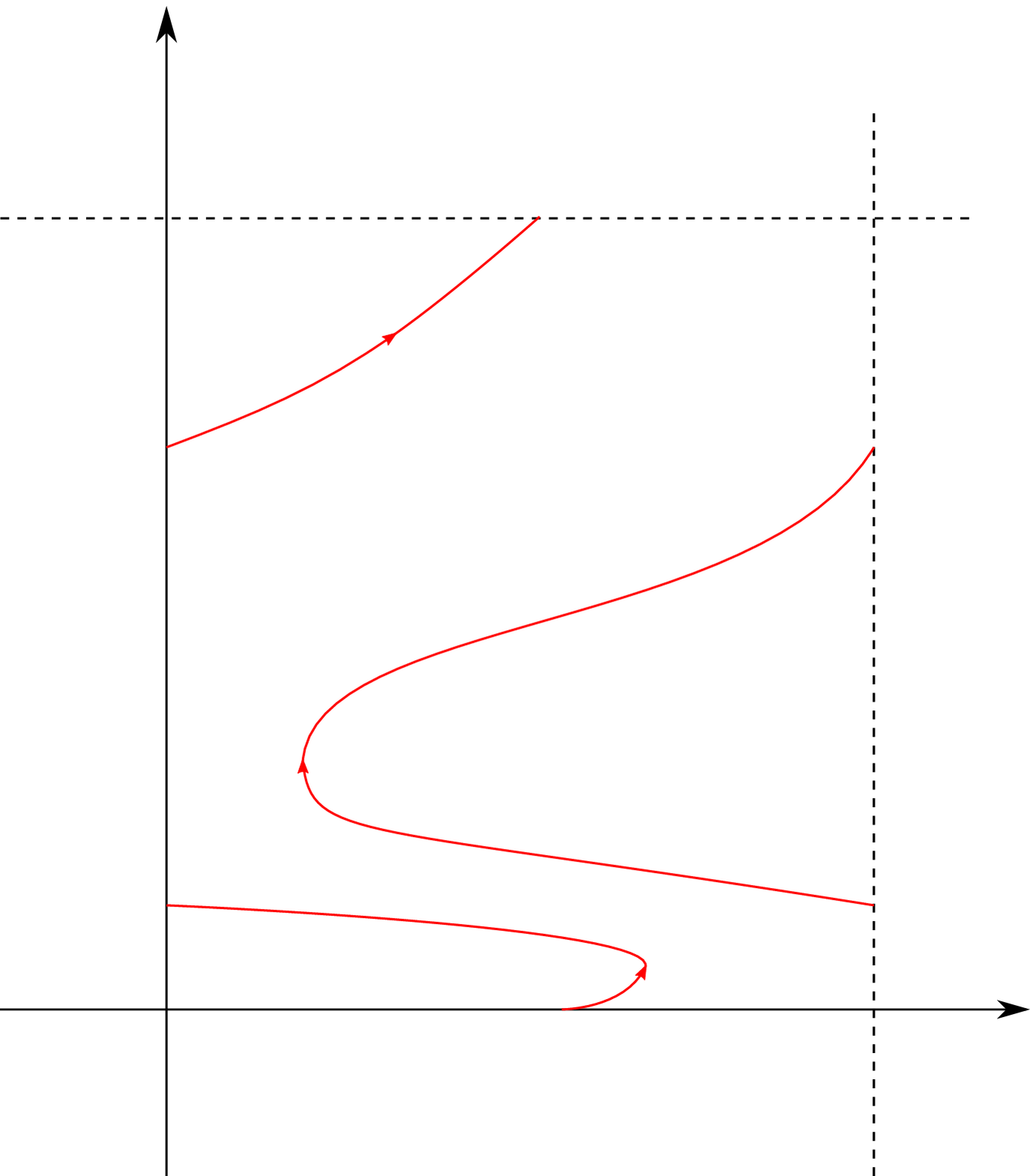
	\end{figure}
	Therefore we concatenate this annulus as many times as such an orbit can  be captured and then we use the filtered Floer homology of this long annulus.
	
	Let us state this more precisely.
	First suppose that $ [l_0] $ is represented by $ [t\mapsto (\lambda_1 t,\lambda_2 t)] $. Take a matrix $ A\in SL(2;\Zbb) $ such that 
	\begin{align*}
	A\begin{pmatrix}
	\lambda_1\\
	\lambda_2
	\end{pmatrix}
	=\begin{pmatrix}
	\lambda\\
	0
	\end{pmatrix}.
	\end{align*}
	Since $ l_0 $ is an embedded loop, $ (\lambda_1,\lambda_2) $ is a primitive element in $ \Zbb^2 $ (cf.  \cite[Proposition 1.5]{FM}) and hence $ \lambda=\pm1 $. By replacing $ A $ by $ -A $ if $ \lambda=-1 $, we may assume that $ \lambda=1 $.
	Using the coordinate chart defined by $ (q',p')^t=A(q,p)^t $, we see that we only need to consider the case $ (\lambda_1,\lambda_2)=(1,0) $.  Note that $ \alpha=(-r,0) $ in this case.

	 Next, we define the Hamiltonian function $ G^k $ on $ M=T^*\Tbb^1 $. We denote by $ \pi: M \to \Tbb^{2} $ the canonical projection $\pi(q,p)= (q,p+\Zbb) $. Put $ \tilde{\alpha}=[t\mapsto(-rt,0)] \in\pi_1(M) $.
	 \begin{df}
	 	For $ k\in \Zbb_+ =\{n\in\Zbb\mid n>0\} $, set $ \nu^k(x):=\mu(|x|-k) $ and define $ G^k\in \Hcal(M) $ by $ G^k(q,p)=\nu^k(p)\pi^*H(q,p) $, where $ \mu $ is the function given by Definition \ref{def:mu}.
	 \end{df}
	 \begin{lem}\label{lem:calc_AH2}
	 	Let $ a,b\in \Rbb $, $ a<b$. There exists $ k_{a,b}\in \Zbb_+ $ such that for every $ k\ge k_{a,b} $ and every  $ x\in P(G^k;\tilde{\alpha}) $, $ \im (x)\sub \{|p|\ge k\} $ holds whenever $ \A_{G^k}(x)\in (a,b) $.
	 \end{lem}
	 We prove this lemma by using the the following four lemmas.
	 \begin{lem}\label{osc}
	 	Let $ x=(q,p):S^1\to M $ be a 1-periodic orbit for $ G^k $. Then we have 
	 	\[ osc(p):=\sup_{s,t\in S^1}|p(s)-p(t)|\le S:=\sup_{S^1\times \Tbb^{2}}\left|\dfrac{\partial H}{\partial q}\right|.\]
	 \end{lem} 
	 \begin{proof}We calculate
	 	\begin{align*}
	 	|p(s)-p(t)|&=\left| \int_{t}^{s}p(\tau)d\tau\right|\\
	 	&=\left|\int_{t}^{s} \dfrac{\partial G^k}{\partial q}(q,p)d\tau\right|\\
	 	&=\left|\int_{t}^{s} \nu^k(p)\dfrac{\partial H}{\partial q}(q,p)d\tau\right|\le S.	
	 	\end{align*}
	 \end{proof}
	 \begin{lem}
	 	Put  $ Q^k_{\pm} =\{ x\in P(G^k;\tilde{\alpha})\mid \im(x)\cap \{(q,p)\in M \mid \pm p\ge k\}\neq \emptyset\}$. Then for any $ k,l\in \Zbb_+ $ with $ k>S $, we have a bijection $ T_l:Q^k_\pm \to Q^{k+l}_\pm$ defined by $ (q,p)\mapsto (q,p\pm l) $.
	 \end{lem}
	 \begin{proof}
	 	We confirm that $ T_l  $ is well defined and surjective.\\
	 	\textbf{(Well-definedness).} Let $ x\in Q^k_\pm $. By Lemma \ref{osc}, $ \im(x)\cap \{\pm p\le  0\}  =\emptyset $ holds. Therefore $T_l(x)$ satisfies the Hamiltonian equation of $ G^{k+l} $, and $ T_ l(x)\in Q^{k+l}_\pm $.\\
	 	\textbf{(Surjectivity).} Let $ x\in Q^{k+l}_\pm  $. By Lemma \ref{osc}, $ \im(x)\cap \{\pm p \le l\} =\emptyset $ holds. Then $ x'=(q, p\mp l,q',p') $ is contained in $ \{ \pm p\ge0 \} $ and hence satisfies the Hamiltonian equation of $ G^k $. Therefore $ x'\in Q^k_{\pm} $ and $ T_l(x')=x $.
	 \end{proof}
	 An easy calculation shows the following lemma.
	 \begin{lem}\label{calc_AH}
	 	$ \A_{G^{k+l}}(T_l(x))=\A_{G^k}(x)\pm rl \qquad   (\forall x \in Q^k_\pm ). $
	 \end{lem}
	 \begin{lem}
	 	For any $ a,b\in\Rbb $ with $ a<b $, there exists $ k_{a,b}\in \Zbb_+ $ such that for every $ k\ge k_{a,b} $, we have
	 	\[ [a,b]\cap \A_{G^{k}}(Q^k_\pm)=\emptyset. \]
	 \end{lem}
	 \begin{proof}
	 	Let $ k>S $. By Theorem \ref{Oh}, $ \A_{H^k}(Q^k_\pm) $ is bounded.  Lemma \ref{calc_AH} implies that if we take $ l\in \Zbb_+ $ sufficiently large, then we have 
	 	\begin{align*}
	 	\inf \A_{G^{k+l}}(Q^{k+l}_+)&=\inf \A_{G^{k+l}}(T_l(Q^{k}_+))=\inf \A_{G^{k}}(Q^{k}_+)+rl >b,\\
	 	\sup \A_{G^{k+l}}(Q^{k+l}_-)&=\sup \A_{G^{k+l}}(T_l(Q^{k}_-))=\sup \A_{G^{k}}(Q^{k}_-)-rl<a.
	 	\end{align*}
	 	This proves the lemma.
	 \end{proof}
	 Now Lemma \ref{lem:calc_AH2} is nothing but a rewording of the above lemma.

	 Let us return to the proof of Theorem \ref{mainthm}. Put $ a':=m_1-1 $ and $ c':=C+S_0+1 $ ($ m_1 $ and $ S_0 $ are defined by \eqref{def:m0S0}).
	 Let $ k $ be a positive integer with $ k>k_{a',c'} $ and define an open convex symplectic manifold by $ M_k= \{ (q,p)\in T^*\Tbb^1 \mid |p|<k+2 \} $. With this notation, $ G^k $ is considered as a Hamiltonian function on $ M_k $. 
	 Let $ \Psi:S^1\times [0,R]\to \Tbb^2 $ be a symplectic embedding given by Lemma \ref{lem:construct embedding}. Take a lift $ \tilde{\Psi}:S^1\times [0,R]\to M_k $ of $ \Psi $. By taking $ k $ sufficiently large, we may assume $ \im(\tilde{\Psi})\sub M_k $ and that $ G^k $ satisfies the condition \eqref{eq:open condition} on $ M_k $. By Theorem \ref{general statement}, we obtain a 1-periodic orbit $ x\in P(G^k;\tilde{\alpha}) $ with $ a<\A_{G^k}(x)<c $. By \eqref{def:a} and \eqref{def:c}, 
	 \begin{align}\label{eq:ab estimate}
	 a'= m_1-1< a<\A_{G^k}(x)< c< C+S_0+1=c'. 
	 \end{align}
	 By this inequality and Lemma \ref{lem:calc_AH2} with $ a $ and $ b $ replaced by $ a' $ and $ c' $, we have $ \im(x)\sub\{|p|<k\} $. Since $ G^k $ is equal to $ \pi^*H $ on the domain $ \{|p|<k\} $,  $ \pi\circ x $ is a 1-periodic orbit for $ H $ in class $ \alpha $, which proves Theorem \ref{mainthm} (1).
	 
	 Now, let us move on to the proof of (3). Suppose that $ H $ is non-degenerate. The following argument was suggested by Tomohiro Asano. First, we prove $ \sharp P(H;\alpha)\ge 3 $.
 	 Here, we consider a relative grading on the filtered Floer homology induced by the Conley-Zender index $ \mu_{CZ} $. Since every map in Theorem \ref{CFH} and \ref{Pozniak} preserves the relative grading, we observe that $ \HFind{i}{G^k} $ and $ \HFind{i+1}{G^k} $ has a direct summand isomorphic to $ H_0(S^1;\Zbb/2) $ and $ H_1(S^1;\Zbb/2) $ respectively. We consider the following long exact sequence.
 	 \begin{align*}
 	 \cdots\to &\HFind{i+2}{G^k}\to HF^{\halfopen{a,c}}_{i+2}{(G^k;\tilde{\alpha})}\to \HFpind{i+2}{G^k}\to\\
 	 \to &\HFind{i+1 }{G^k}\to HF^{\halfopen{a,c}}_{i+1 }{(G^k;\tilde{\alpha})}\to \HFpind{i+1 }{G^k}\to\\
 	 \to &\HFind{i}{G^k}\to HF^{\halfopen{a,c}}_{i}{(G^k;\tilde{\alpha})}\to \HFpind{i}{G^k}\to\cdots
 	 \end{align*}
 	 By Theorem \ref{CFH} (3), we have $  HF^{\halfopen{a,c}}_{*}{(G^k;\tilde{\alpha})}=0 $. We conclude that  $ \HFpind{i+2}{G^k} $ and $ \HFpind{i+1}{G^k} $ also have a direct summand isomorphic to $ \Zbb/2 $. Therefore, we obtain at least three 1-periodic orbits $ x_i, x_{i+1}, x_{i+2}\in P(G^k;\tilde{\alpha}) $
 	 with $ \mu_{CZ}(x_j)=j\ (j=i,i+1,i+2) $. Since the projection $ \pi:M_k\to \Tbb^2 $ preserves the indices of orbits, we obtain  three distinct 1-periodic orbits $ x_{i}, x_{i+1}, x_{i+2}\in P(H;\alpha) $. 
 	 
 	 Next, we prove that $ \sharp P(H;\alpha)  $ is an even number, which implies that $ \sharp P(H;\alpha)\ge 4 $. 
 	 There is a one-to-one correspondence between the intersection of $ G=\mathrm{graph}(\phi_H^1) $ and $ \Delta $ in $ \Tbb^2 $ and 1-periodic orbits for a Hamiltonian function on $ \Tbb^2 $  by the map $ (x_0,x_0)\mapsto \{\phi_H^t(x_0) \} $, where $ \Delta $ refers to the diagonal subset in $ \Tbb^4 $. If $ \{x(t) \}\in \alpha $ then $ \tilde{x}(1)-\tilde{x}(0)=\alpha $, where $ \tilde{x} $ is a lift of $ x $ to $ \Rbb^4 $. Let $ \psi:\Tbb^4\to \Tbb^4 $ be a map defined by $ \psi(q,p,Q,P)=(q,p,Q-q,P-p) $. Then $ \psi(\Delta)=\Tbb^2\times\{0\} $ and $ \psi(G)=\{(q,p,\phi_H^1(q,p)-(q,p))\mid (q,p)\in \Tbb^2 \} $. Let $ \pi':\Tbb^2\times \Rbb^2\to \Tbb^4 $ be a projection in the third and forth variables and $ \tilde{f}_t:\Tbb^2\to \Tbb^2\times\Rbb^2 $ be the lift of the homotopy $ \{f_t:\Tbb^2\to\Tbb^4 \};\ f_t(q,p)=\psi(q,p,\phi_H^t(q,p)) $  with $ \tilde{f}_0(\Tbb^2 ) =\Tbb^2\times\{0\}$. Put $ \tilde{G}=\tilde{f}_1(\Tbb^2) $ and $ \tilde{\Delta}=\Tbb^2\times\{(\alpha_1,\alpha_2) \} $. Then $ \tilde{G}\cap \tilde{\Delta} $ corresponds to each element in $ P(H;\alpha) $. In fact, take $ w\in \tilde{G}\cap \tilde{\Delta} $. Then $ \psi^{-1}\circ\pi'(w)\in G\cap \Delta $, which means $ \psi^{-1}\circ\pi'(w)=(x_0,x_0) $ and $ \phi_H^1(x_0)=x_0 $. Put $ x(t)=\phi_H^t(x_0) $. Then $ x $ is a 1-periodic orbit for $ H $.  Since $ \tilde{f}_t(x_0)=(x_0,\tilde{x}(t)-\tilde{x}(0)) $ and $ (0,\alpha)=w-(x_0,0)=\tilde{f}_1(x_0)-\tilde{f}_0(x_0) $, we have  $ \alpha=\tilde{x}(1)-\tilde{x}(0) $, which implies that $ [x]=\alpha $.
 	 The converse is given by a similar argument.
 	 This correspondence implies   
 	 \[ \sharp \tilde{G}\cap \tilde{\Delta}=\sharp P(H;\alpha) . \]
 	 Now, since $ \tilde{G} $ is homotopic to $ \Tbb^2\times\{0\} $, the intersection number of two subsets $ \tilde{G}, \tilde{\Delta}\sub \Tbb^2\times \Rbb^2 $ is zero. In addition, since $ H $ is non-degenerate, $ \tilde{G} $ and $ \tilde{\Delta} $ intersects transversally. Therefore,  $ \sharp \tilde{G}\cap \tilde{\Delta} $ is an even number.
\end{proof}

\section{Examples}
First, we consider Remark \ref{rem:product by aspherical mfd} in  the case $ \Sigma=\Tbb^2=\Rbb^2/\Zbb^2 $, $ N=\Tbb^{2n-2} $, $ l_0(t)=(t,0) $ and $ l_1(t)=(t,R)\ (0<R<1) $. Then we have 
\[  C(\Tbb^{2n}(q,p),\{p_1=0\},\{p_1=R\},[t\mapsto(-rt,0,\cdots,0)])=rR. \]
If $ n=1 $, this contains the result of Biran, Polterovich and Salamon \cite{BPS} in the case of an annulus. If $ n\ge 2 $, $ X $ and $ Y $ are codimension 1 submanifold.
However, it seems to be natural to take $ X $ and $ Y $ to be Lagrangian tori. In this case, all arguments given in the proof of Theorem \ref{mainthm} for $ \Tbb^2 $ go through except for the construction of monotone homotopies. In fact, we can prove that the capacity is infinite. 
\begin{prp}\label{prp:cacity infinite}
	For every non-trivial homotopy class $ \alpha\in \pi_1(\Tbb^{n})\times\{0\}\subset\pi_1(\Tbb^{2n}) $ and $ w\in\Tbb^n\setminus\{0\} $,
	\[ C(\Tbb^{2n},\{p=0 \},\{p=w \};\alpha)=\infty. \]
\end{prp}
\begin{proof}
	Choose $ \beta\in \Zbb^n $ with $ \beta\cdot w \neq 0 $ and  $ \alpha\neq r\beta $ for any $ r\in\Rbb $. Define $ \Gamma\sub \Tbb^{n} $ by $ \Gamma=\{p\in \Tbb^n\mid (p-w)\cdot \beta=0 \} $ and define a Hamiltonian function $ H_k\in \Hcal(\Tbb^{2n}) $ by $ H_k(q,p)=f_{0,-k,d(0,\Gamma)/2}(d(p,\Gamma)) $, where $ k\in \Zbb_+ $ and $ f $ is given by \eqref{def:f}. Then we have $ P(H_k;\alpha)=\emptyset$ and $ \inf_X H_k -\sup_Y H_k=k  $. Indeed, the non-constant solutions of the Hamiltonian equation are written as 
	\[ (q(t),p(t))=(q^0+tr\beta,p^0), \]
	for some $ r\in \Rbb$ , $ q^0\in\Tbb^n $ and $ p^0\in\Tbb^n $. By the assumption that $ \alpha\neq r\beta $, we have $ P(H_k;\alpha)=\emptyset$.
\end{proof}
As a related result, Polterovich constructed a Hamiltonian function on $ \Tbb^4 $ with a non-standard symplectic structure whose periodic orbits are all constant loops (see \cite[Example 1.2]{Pol}).

We would like to pose the following problem (in fact, this is considered as a generalization of \cite{BPS} in the case of unit cotangent bundle of a torus).
\begin{problem}
	Let $ X=\{(q,p)\mid p=0\} $, $ Y=\{(q,p)\mid p\in \Lambda \} $ and $ \alpha\in \pi_1(\Tbb^{n})\times\{0\} $, where $ \Lambda:=\{p\in\Tbb^n\in |p|=R \} \ (0<R<1/2)$. Then we have
	\[ C(\Tbb^{2n},X,Y;\alpha)=|\alpha|R=R\sqrt{\alpha_1^2+\cdots+\alpha_n^2}.  \]
	
\end{problem}
By an argument similar to the proof of Theorem \ref{mainthm}, we can confirm that there exists a 1-periodic orbit  for every  Hamiltonian function with $ \inf_{S^1\times X}H-\sup_{S^1\times Y}  H>|\alpha|R $ and $  \sup_{S^1\times M} H<|\alpha| R+\inf_{S^1\times X} H $. This fact partially supports the above statement.

\section{Basic properties of $ C(M,X,Y;\alpha) $}
In this section, we summarize basic properties of the generalized BPS capacity. 
Most of the proofs are the same as in \cite{BPS}. Assume that $ (M,\omega) $ is a connected symplectic manifold.
\begin{prp}[{Monotonicity \cite[Proposition 3.3.1]{BPS}}]
	Let $ M' $ be an open subset of $ M $ and $ X,Y\sub M' $ be compact subsets. Suppose that $ \iota_*:\pi_1(M')\to\pi_1(M) $ is injective, where $ \iota $ is the inclusion $ M'\to M $. We regard $ M' $ as a symplectic manifold with symplectic form $ \iota^*\omega $. Then 
	\begin{align*}
	C(M',X,Y;\alpha)\le C(M,\iota(X), \iota(Y);\iota_*(\alpha)).
	\end{align*} 
\end{prp}

\begin{lem}\label{lem:loop homotopy}
	Let $ F,G\in \Hcal(M) $ and put $ \phi_t=\phi^t_F $ and $ \psi_t=\phi^t_G $. Suppose that
	\begin{enumerate}[$ \bullet $]
		\item $ \phi_1=\psi_1 $,
		\item $ \exists z\in M$, two paths $\phi_t(z) $ and $ \psi_t(z) $ are homotopic relative to endpoints.
	\end{enumerate}
	Then for every $ x\in M $, $\phi_t(x) $ and $ \psi_t(x) $ are homotopic relative to end points.
\end{lem}
\begin{proof}
	Choose a path $ \gamma:S^1\to M $ such that $ \gamma(0)=x $ and $ \gamma(1)=z $. Let $y\in M $ and defined $\{ h_t(y)\} $ to be the concatenation of the two paths $ \{\phi_t(y)\}$ and $ \{\psi_{1-t}(y)\}$. Then $ \{h_t(\gamma(s))\}_{s\in[0,1]} $ gives a homotopy between two loops $ \{\phi(x)\sharp(\psi(x))^{-1}\}  $  and $ \{\phi(z)\sharp(\psi(z))^{-1}\} $. By assumption,  $ \{\phi(z)\sharp(\psi(z))^{-1}\} $ is contractible and so is $ \{\phi(x)\sharp(\psi(x))^{-1}\}  $, which completes the proof.
\end{proof}
Note that the second condition is always satisfied if $ M $ is open.

\begin{prp}[{Displacement \cite[Proposition3.3.2]{BPS}}]
	Let $ X,Y\sub M $ be compact subsets and $ \alpha\in\pi_1(M) $ a non-trivial class. Suppose that there exists a Hamiltonian function $ F\in \Hcal(M) $ such that $ \phi^1_F(X)\cap X=\emptyset $, $ P(F;\alpha)=\emptyset $ and $ \phi^t_F(z)=z\ (0\le t\le1) $ for some $ z\in M $. Then $ C(M,X,Y;\alpha)=\infty $. The same is true for $ Y $.
\end{prp}
\begin{proof}
	We only prove in the case of  $ X $ (the proof for $ Y $ is the same).  Fix $ c>0 $ and take an open neighborhood $ U $ of $ X $ with $ z\notin U $ and $ \phi^1_F(U)\cap U= \emptyset $.
	Choose $ G\in\Cinfc(M) $ such that  
	\begin{align*}
	\supp(G)\sub U, \quad \inf_X G +\inf_{S^1\times X} F-\sup_{S^1\times Y} F\ge c.
	\end{align*}
	Put $ f_t=\phi_F^t,\ g_t=\phi_G^t $ and $ h_t=f_t\circ g_t $. Since $ f_1(U)\cap U= \emptyset $, for every $ x\in \Fix(f_1) $, we have $ x\notin U $ and $ [\{h_t(x)\}]= [\{f_t(x)\}]\ne \alpha $. In addition, $ \Fix(h_1)=\Fix(f_1) $ holds. Set $ \phi_t=(f_1)^{-1}\circ h_t\circ f_1 $ and $ \psi_t=g_t\circ f_t $. Note that $ \phi_1=g_1 \circ f_1=\psi_1 $ and that there is a one-to-one correspondence between $ \Fix(h_1) $ and $ \Fix(\phi_1) $ given by $ x\mapsto (f_1)^{-1}(x) $.  Since $ \phi_t(z)=z=\psi_t(z) $, by Lemma \ref{lem:loop homotopy}, $ \{\phi_t(x)\} $ and $\{ \psi_t(x)\} $ have the same homotopy type for every $ x\in \Fix(\psi_1) $. Therefore $ [\{\psi_t(x)\}]=[\{\phi_t(x)\}]=[\{h_t(x)\}]\ne \alpha $ for every $ x\in\Fix(\psi_1) $. Since $ \psi_t $ is the flow of the Hamiltonian function $ H_t(y)= G(y)+F_t( (g_t)^{-1}(y) ) $ and $ \inf_{S^1\times X} H - \sup_{S^1\times Y} H\ge c $ holds, we have $ C(M,X,Y;\alpha)\ge c $. 
\end{proof}

\begin{prp}[{\cite[Proposition 3.3.4.]{BPS}}]\label{prp:convergence periodic orbits}
	Let $ X,Y\sub M $ be compact subsets and $ \alpha\in\pi_1(M) $. Suppose that $ \inf_{S^1\times X}H-\sup_{S^1\times Y} H\ge C(M,X,Y;\alpha) $. Then $ P(H;\alpha)\ne \emptyset $.
\end{prp}

\section{Comparison to invariant measures}
Kawasaki defined in \cite{Ka} another relative capacity using invariant measures of a Hamiltonian flow.  We start with a review of  basic terminology. 
Let $ \mu $ be a  Borel probability measure on $ M $. We say $ \mu $ is invariant with respect to a family of homeomorphism  $ \{\phi_t\}_{t\in [0,1]}  $ if $ \mu(\phi_t^{-1}(A))=\mu(A) $ for every measurable subset $ A $ and $ t\in[0,1] $. Let $ X $ be a vector field on $ M $. Denote by $ \meas(M,X) $ the set of Borel probability measures invariant with respect to the flow $ \phi_t $ of $ X $. 
Take $ \mu\in\meas(M,X) $ and define a first homology class $ \rho(\mu, X)\in H_1(M;\Rbb) $, called the rotation vector, by 
\begin{align*}
\langle [\alpha],\rho(\mu,X)\rangle=\int_M \alpha(X)\mu,
\end{align*}
for all closed 1-forms $ \alpha $. 
\begin{df}
	Let $ A $ and $ B $ be compact subsets of $ M $ and  $ \mathfrak{l}\in H^1(M;\Rbb) $.
	We define a relative symplectic capacity $ C^P(M,A,B;\mathfrak{l},\alpha) $ by 
	\begin{align*}
	C^P(M,A,B;\mathfrak{l},\alpha):=\inf\{c>0\mid \forall H\in \Cinfc(M)\text{ with }\inf_{A} H -\sup_{B}H\ge c,\\ \exists \mu\in\meas(M,X_H)\text{ such that }|\langle \mathfrak{l},\rho(\mu, X_H)\rangle|\ge \mathfrak{l}(\alpha) \}.
	\end{align*}
	Put 
	\begin{align*}
	C^P(M,A,B;\alpha):=\sup_{\mathfrak{l}\in H^1(M;\Rbb)} C^P(M,A,B;\mathfrak{l},\alpha).
	\end{align*}
\end{df}
\begin{prp}\label{prp:CP le C}
	With the above definitions, we have
	\begin{align*}
	C^P(M,A,B;\alpha)\le C(M,A,B;\alpha).
	\end{align*}
\end{prp}
\begin{proof}
	Let $ c> C(M,A,B;\alpha) $. Suppose that $ H\in \Cinfc(M) $ satisfies $ \inf_{A} H -\sup_{B}H\ge c $. Then there is a 1-periodic orbit $ x\in P(H;\alpha) $. Put $ \mu:=x_*(dt) $, that is, the  pushforward measure of Lebesgue measure on $ S^1 $ by $ x $. This is invariant with respect to $ \phi_H^t $ and $ \rho(\mu,X_H)=\bar{\alpha} $, where $ \bar{\alpha} $ is the image of $ \alpha $ by the Hurewicz homomorphism. Therefore, we have $ c\ge C^P(M,A,B;\alpha) $.
\end{proof}
Proposition \ref{prp:cacity infinite} or Example 1.2 in \cite{Pol} together with the theorem below gives an example where $ C^P(M,A,B;\alpha)=C(M,A,B;\alpha) $ does not hold.
\begin{thm}[\cite{Pol}]
	Let $ A $ be a  compact subset of a closed symplectic
	manifold $ (M,\omega) $.  Assume that $ A $ is non-displaceable, i.e., $ A\cap \phi(A)\ne\emptyset $ by any Hamiltonian diffeomorphism $ \phi\in \mathrm{Ham}(M,\omega) $ and that  there exists a symplectic isotopy $ \{\phi_t \}_{t\in[0,1]} $ such	that $ A\cap \phi_1(A)=\emptyset $. Put $ \mathfrak{l}=\mathrm{Flux}(\{\phi_t \}) $ and $ B=\phi(A) $. Then for any positive real number $ p $ and any Hamiltonian
	function $ F: M\to \Rbb $ such that $ \inf_A F-\sup_B F\ge p $, there exists a Borel probability measure measure $ \mu $ invariant with respect to $ \phi_F^t $ such that $ \supp(\mu)\sub\supp(F) $ and
	\begin{align*}
	|\langle \mathfrak{l}, \rho(\mu,X_F)\rangle| \ge p.
	\end{align*}
\end{thm}
As a corollary of the above theorem, we have
\begin{align*}
C^P(M,A,B;\mathfrak{l},\alpha)\le |\mathfrak{l}(\alpha)|.
\end{align*}
Applying this inequality to $ \Tbb^{2n} $, we have $ C^P(\Tbb^{2n},\{p=0\},\{p=w\};w,\alpha)\le|\sum_{i=1}^n w_i\alpha_i| $.
This result contrasts to Proposition \ref{prp:cacity infinite}.


%


{\small (Hiroyuki Ishiguro) \textsc{Graduate School of Mathematical Sciences, the University of Tokyo, 3-8-1 Komaba, Meguro-ku, Tokyo 153-0041, Japan}\\
\textit{Email address}: \texttt{ishiguro@ms.u-tokyo.ac.jp} or \texttt{hryk1496@gmail.com}}

\end{document}